\documentclass{amsart}

\usepackage[utf8]{inputenc}
\usepackage[english]{babel}

\usepackage{fullpage}

\usepackage{amsmath}
\usepackage{amsfonts}
\usepackage{amsthm}
\usepackage{amssymb}
\usepackage{theoremref}
\usepackage{colonequals}
\usepackage{url}
\usepackage[inline]{enumitem}
\usepackage{csquotes}
\usepackage{nameref}

\newcounter{master}
\numberwithin{master}{section}

\theoremstyle{plain}

\newtheorem{theorem}[master]{Theorem}
\newtheorem{problem}[master]{Problem}
\newtheorem{proposition}[master]{Proposition}
\newtheorem{lemma}[master]{Lemma}
\newtheorem{corollary}[master]{Corollary}
\newtheorem{conjecture}[master]{Conjecture}
\newtheorem{question}[master]{Question}
\newtheorem{claim}[master]{Claim}
\newtheorem{intro}{Theorem}
\newtheorem*{custom}{\customthname}

\theoremstyle{definition}

\newtheorem{definition}[master]{Definition}
\newtheorem*{notation}{Notation}

\theoremstyle{remark}

\newtheorem*{remark}{Remark}
\pretocmd{\endremark}{\hfill$\blacktriangleleft$}{}{}

\makeatletter
\let\c@equation\c@master

\let\c@figure\c@master

\let\c@table\c@master

\makeatother

\let\oldcustom\custom
\def\custom#1{\def\customthname{#1}\oldcustom}

\makeatletter
\def\@cite#1#2{\textup{[\textbf{#1}\if@tempswa , #2\fi]}}
\def\@biblabel#1{[\textbf{#1}]}
\makeatother

\DeclareMathOperator{\Const}{const.}
\DeclareMathOperator{\Diag}{diag}
\DeclareMathOperator{\End}{End}
\DeclareMathOperator{\Gr}{Gr}
\DeclareMathOperator{\Graff}{Graf{}f}
\DeclareMathOperator{\Int}{int}
\DeclareMathOperator{\Stab}{Stab}
\DeclareMathOperator{\Vol}{vol}

\let\O\undefined
\DeclareMathOperator{\B}{B}
\DeclareMathOperator{\D}{D}
\DeclareMathOperator{\GA}{GA}
\DeclareMathOperator{\GL}{GL}
\DeclareMathOperator{\O}{O}
\DeclareMathOperator{\SO}{SO}

\makeatletter\def\tr#1{{\reset@font[}#1{\reset@font]}}\makeatother

\usepackage{tikz}
\usetikzlibrary{matrix, decorations.pathreplacing, calligraphy, calc, arrows.meta}

\def\backmatter{\def\subsection##1{\def\thesubsubsection{\S\ref{##1}}\subsubsection{\nameref{##1}}}}

\def\sref{Section~\S\ref}
\def\fref{fig.~\ref}

\begin{document}

\title{On Bezdek's conjecture for high-dimensional convex bodies with an aligned center of symmetry}

\author[M.A. Alfonseca]{M. Angeles Alfonseca}
\address[M.A. Alfonseca]{North Dakota State University, Fargo, ND, USA}
\email{maria.alfonseca@ndsu.edu}
\author[B. Zawalski]{Bartłomiej Zawalski}
\address[B. Zawalski]{Kent State University, Kent, OH, USA}
\email{bzawalsk@kent.edu}
\thanks{The first author is supported in part by the Simons Foundation gift MPS-TSM-00711907. The second author is supported in part by U.S. National Science Foundation Grants DMS-1900008 and DMS-2247771.}

\begin{abstract}
In 1999, K.~Bezdek posed a~conjecture stating that among all convex bodies in $\mathbb R^3$, ellipsoids and bodies of revolution are characterized by the fact that all their planar sections have an axis of reflection. We prove Bezdek's conjecture in arbitrary dimension $n\geq 3$, assuming only that sections passing through a~fixed point have an axis of reflection, provided that the complementary invariant subspaces are all parallel to a~fixed hyperplane. The result is proven in both orthogonal and affine settings.
\end{abstract}

\maketitle

\tableofcontents

\section{Introduction}

In 1889, H.~Brunn \cite{brunn1889curven} proved that among all convex bodies, ellipsoids are characterized by the fact that all their flat sections are centrally symmetric. Brunn's result was improved in 1971 by P.W.~Aitchinson, C.M.~Petty, and C.A.~Rogers \cite{Aitchison_Petty_Rogers_1971}, and in 1974 by D.G.~Larman \cite{Larman}, who proved the False Center Theorem:

\begin{custom}{False Center Theorem}[cf. \thref{thm:11}]
Let $K\subset\mathbb R^n$, $n\geq 3$, be a~convex body. If all sections of $K$ by hyperplanes passing through a~fixed point which is not the center of symmetry of $K$ are centrally symmetric, then $K$ is an ellipsoid.
\end{custom}

\noindent In 2007, L.~Montejano and E.~Morales-Amaya \cite{S002557930000019X} further relaxed the assumption that all hyperplanes pass through a~fixed point, proving the Shaken False Center Theorem:

\begin{custom}{Shaken False Center Theorem}[cf. \thref{thm:04}]
Let $K\subset\mathbb R^3$ be an origin-symmetric convex body and let $\delta:\mathbb S^2\to\mathbb R$ be an even continuous map which is $C^1$ in a~neighbourhood of $\delta^{-1}(\{0\})$. Denote $H_{\boldsymbol\xi}^\delta\colonequals\langle\boldsymbol\xi\rangle^\perp+\delta(\boldsymbol\xi)\boldsymbol\xi$, $\boldsymbol\xi\in\mathbb S^2$. If for every $\boldsymbol\xi\in\mathbb S^2$, either $\delta(\boldsymbol\xi)=0$ and $K\cap H_{\boldsymbol\xi}^\delta$ is an ellipse or $\delta(\boldsymbol\xi)\neq 0$ and $K\cap H_{\boldsymbol\xi}^\delta$ is centrally symmetric, then $K$ is an ellipsoid.
\end{custom}

In addition to being centrally symmetric, all hyperplane sections of ellipsoids also possess an axis of symmetry, as do the hyperplane sections of general bodies of revolution. In 1999, K.~Bezdek \cite{Odor1999} posed the following conjecture:

\begin{custom}{Bezdek's Conjecture}[cf. \thref{con:02}]
Let $K\subset\mathbb R^3$ be a~convex body. If all sections of $K$ by affine planes admit an axis of symmetry, then $K$ is either a~body of revolution or an ellipsoid.
\end{custom}

In this paper, we formulate and prove several higher-dimensional analogs of Bezdek's conjecture and its affine counterpart. We say that a~convex body admits an affine symmetry if it is the image under an affine transformation of a~convex body that admits the symmetry in the usual sense. Our main results are as follows:

\begin{intro}[cf. \thref{thm:14}]\thlabel{intro:01}
Let $K\subset\mathbb R^n$, $n\geq 3$, be a~convex body and let $p\in\mathbb R^n$. If all hyperplane sections $K\cap H$ passing through $p$ are invariant under (affine) reflection through a~$1$-dimensional subspace of $H$ passing through $p$, and the complementary invariant subspaces are all parallel to a~fixed hyperplane, then $K$ is a~body of (affine) revolution, with axis of revolution passing through $p$, or an ellipsoid.
\end{intro}

\begin{intro}[cf. \thref{thm:18}]\thlabel{intro:02}
Let $K\subset\mathbb R^n$, $n\geq 3$, be a~convex body and let $p\in\mathbb R^n$. If all hyperplane sections $K\cap H$ passing through $p$ are invariant under (affine) rotations about a~$1$-dimensional subspace of $H$ passing through $p$, and the complementary invariant subspaces are all parallel to a~fixed hyperplane, then $K$ is a~body of (affine) revolution, with the axis of revolution passing through $p$.
\end{intro}

\noindent We also obtain the following two results, where the entire family of hyperplanes passing through $p$ is replaced by its $1$-codimensional subfamily:

\begin{intro}[cf. \thref{thm:16}]\thlabel{intro:03}
Let $K\subset\mathbb R^4$ be a~strictly convex body, let $p\in\Int K$ be any point in the interior of $K$, let $T\in\Gr_3(\mathbb R^4)$ be any hyperplane and let $\sigma:\Gr_2(T)\to\Gr_3(\mathbb R^4)$ be a~Lipschitz continuous map such that $\sigma(H)\cap T=H$ for all $H\in\Gr_2(T)$. If every intersection $K\cap(\sigma(H)+p)$ is invariant under affine reflection through a~$1$-dimensional subspace passing through $p$, with complementary invariant subspace $H+p$, then $K$ is a~body of affine reflection, with the axis of reflection passing through $p$.
\end{intro}

\begin{intro}[cf. \thref{thm:02}]\thlabel{intro:04}
Let $K\subset\mathbb R^n$, $n\geq 4$, be a~convex body, let $p\in\Int K$ be any point in the interior of $K$, let $T\in\Gr_{n-1}(\mathbb R^n)$ be any hyperplane and let $\sigma:\Gr_{n-2}(T)\to\Gr_{n-1}(\mathbb R^n)$ be a~continuous map such that $\sigma(H)\cap T=H$ for all $H\in\Gr_{n-2}(T)$. If every intersection $K\cap(\sigma(H)+p)$ is invariant under affine revolution about a~$1$-dimensional subspace passing through $p$, with complementary invariant subspace $H+p$, then $K$ is a~body of affine revolution, with the axis of revolution passing through $p$.
\end{intro}

\noindent We also obtain an extension of \thref{intro:01} to the case of orthogonal reflections through $k$-dimensional subspaces for any $1\leq k<n-1$:

\begin{intro}[cf. \thref{thm:22}]\thlabel{intro:05}
Let $K\subset\mathbb R^n$, $n\geq 3$, be a~convex body, and let $p\in\mathbb R^n$ be any point of the ambient space. If all hyperplane sections $K\cap H$ passing through $p$ are invariant under reflection through a~$k$-dimensional subspace of $H$ passing through $p$, and the complementary invariant subspaces are all parallel to a~fixed $k$-codimensional hyperplane, $1\leq k<n-1$, then $K$ is a~body of $k$-revolution, with hyperaxis of $k$-revolution passing through $p$.
\end{intro}

The underlying idea behind the proofs of these theorems can be illustrated by a~simplified three-dimensional case. Let $K\subset\mathbb R^3$ be a~convex body. Suppose there exists a~point $p\in\mathbb R^3$ and a~plane $T$ not containing $p$ such that, for every plane $H$ passing through $p$, the intersection $K\cap H$ admits an axis of affine symmetry through $p$, with the complementary invariant subspace parallel to $T$. Then, for any plane $T'$ parallel to $T$, the intersection of the axis of symmetry with $T'$ coincides with the center of the chord $K\cap T'\cap H$. Projecting all parallel sections $K\cap T'$ from $p$ onto a~single plane $T$ produces a~continuous family of convex bodies in $T$ whose intersections with any line in $T$ share a~common center. Such a~configuration is possible only if the projections are concentric, homothetic ellipses, which in turn implies that $K$ is a~body of affine revolution with respect to the axis through $p$ and the common center of these projections. In higher dimensions, the second part of the argument becomes substantially more intricate. Nevertheless, the point of intersection of the axis of affine symmetry with $K\cap T'\cap H$ is not merely the center of mass but the center of symmetry of $K\cap T'\cap H$. By classical results such as Brunn's theorem, one can therefore assume from the outset that $K\cap T'$ is an ellipsoid. This observation significantly simplifies the higher-dimensional argument, reducing it essentially to linear algebra. We elaborate on this paradigm in detail in \sref{sec:07}, after introducing the necessary terminology.\\

The motivation of this paper is three-fold. First, we prove several new theorems which lay the foundations for generalizing Bezdek's conjecture to higher dimensions. Since every group of symmetries of a~convex body contains a~group of reflections as a~subgroup, the case we address is the most basic and the most difficult at the same time. Secondly, we introduce a~unified language in which results such as Brunn's theorem, the False Center Theorem, Bezdek's conjecture, and many more can be easily seen as special cases of a~more general problem. Finally, we survey in detail many historical results related to the symmetries of sections of convex bodies, providing readers with a~comprehensive overview of the solved problems and open questions in this area.\\

The structure of the paper is the following: In \sref{sec:10}, we introduce the definitions of group invariance, the action of the orthogonal group by reflections and rotations, and the notions of pseudo- and quasi-center of symmetry, which we will need to formulate our results. In \sref{sec:01}, we give a~thorough historical exposition of numerous theorems in geometric tomography and beyond, that may be viewed as specific instances of the general problem we consider. In \sref{sec:02}, we introduce the key notion of an \emph{aligned center of symmetry} (see \thref{def:03,def:05}), distilled from the ideas present in \cite{Alfonseca} and \cite{faor}. We also outline the general paradigm that guides our work and formulate our new results. In \sref{sec:06}, we present a~general argument that allows us to deduce the orthogonal variants of some of our inherently affine theorems. In \sref{sec:04}, we give the complete proofs of all the new results. Meanwhile, we also state several open questions that we ourselves find interesting, and partially answer some of them right away. We conclude the paper with some tangential comments in \sref{sec:11}.

\section{Definitions and basic concepts}\label{sec:10}

We denote by $\mathbb R^n$, $n\geq 2$, the $n$-dimensional Euclidean space equipped with the standard inner product $\cdot$ and induced norm $\|\cdot\|$. Euclidean space carries the structure of both a~linear and an affine space. For any affine subspace $V\subset\mathbb R^n$, we denote its associated vector space by $\vec V$. The (affine) span of a~set $S\subset\mathbb R^n$ is denoted $\langle S\rangle$. Also, we denote by $\mathbb B^n$ the Euclidean unit ball in $\mathbb R^n$ and by $\mathbb S^{n-1}$ the Euclidean unit sphere in $\mathbb R^n$. A~\emph{convex body} $K\subset\mathbb R^n$ is a~compact convex set with a~non-empty interior. By $\Gr_k(\mathbb R^n)$, $1\leq k<n$, we denote the \emph{Grassmannian} of all $k$-dimensional linear subspaces of $\mathbb R^n$ and by $\Graff_k(\mathbb R^n)$, $1\leq k<n$, we denote the \emph{affine Grassmannian} of all $k$-dimensional affine subspaces of $\mathbb R^n$. The Grassmannian comes equipped with a~canonical orthogonally invariant metric space structure (see \cite[\S 4]{Wong}).\\

Let $K\subset\mathbb R^n$, $n\geq 2$, be a~convex body and let $G\leq\GA(\mathbb R^n)$ be a~compact subgroup of volume-preserving affine automorphisms of the ambient space. We say that $K$ is \emph{invariant under the action of $G$} if $gK=K$ for every $g\in G$. Clearly, the group $G$ acts on $\Graff_k(\mathbb R^n)$. For an affine subspace $H\in\Graff_k(\mathbb R^n)$, consider $A=\Stab_G(H)=\{g\in G\mid gH=H\}$. Since both $K$ and $H$ are invariant under the action of $A$, so is the intersection $K\cap H$. Besides, $gH$ is an invariant subspace of $gAg^{-1}$ for every $g\in G$. This gives rise to the entire family $\mathcal O\subseteq\Graff_k(\mathbb R^n)$ of affine subspaces such that for every $gH\in\mathcal O$, the intersection $K\cap gH$ is invariant under the action of the group $gA\vert_Hg^{-1}$, which is affinely conjugate to $A\vert_H$ by $g\in G$. The orbit-stabilizer theorem (see \cite[Lemma~2.18]{kirillov2017introduction}) yields
$$\dim G-\dim A=\dim\mathcal O,$$
which means that the family $\mathcal O$ must be considerably large. Hence, if a~convex body $K$ is invariant under the action of a~group $G\leq\GA(\mathbb R^n)$, many of its sections by affine subspaces of dimension $k$ are themself invariant under the action of a~group affinely conjugate to a~fixed subgroup $A\leq\O(k,\mathbb R)$.\\

One notable example is a~convex body $K\subset\mathbb R^n$, $n\geq 2$, invariant under the action of a~group isomorphic to $\O(n-k,\mathbb R)$, generated by reflections through affine hyperplanes containing a~fixed $k$-dimensional affine subspace $V^\perp+a\in\Graff_k(\mathbb R^n)$, $0\leq k<n$. Then the intersection $K\cap(H+b)$ of $K$ with any affine subspace $H+b$ is invariant under the action of a~group generated by reflections through affine hyperplanes containing $(V\cap H)^\perp+b$, isomorphic to $\O(\dim(V\cap H),\mathbb R)$. In general, $K\cap(H+b)$ does not possess any other symmetries.\\

Another example is a~convex body $K\subset\mathbb R^n$, $n\geq 2$, invariant under the action of a~group isomorphic to $\O(1,\mathbb R)$, generated by the reflection through a~fixed $k$-dimensional affine subspace $V^\perp+a\in\Graff_k(\mathbb R^n)$, $0\leq k<n$. Then every affine subspace $H_1+b$ contains a~(possibly trivial) affine subspace $H_2+b\colonequals(V^\perp\cap H_1)\oplus(V\cap H_1)+b$ such that the intersection $K\cap(H_2+b)$ is invariant under the action of a~group generated by the reflection through $(V^\perp\cap H_2)+b$ (again, isomorphic to $\O(1,\mathbb R)$). In general, $K\cap(H_i+b)$, $i=1,2$, does not possess any other symmetries.\\

It is interesting to question whether we may reverse these implications. In other words, can we deduce the symmetry group of a~convex body $K\subset\mathbb R^n$, $n\geq 2$, knowing that a~considerably large family of its sections by $k$-dimensional affine subspaces is invariant under the action of a~group affinely conjugate to a~fixed subgroup $A\leq\O(k,\mathbb R)$? The most general problem of this type may be formulated as follows:

\begin{problem}\thlabel{pro:01}
Let $K\subset\mathbb R^n$, $n\geq 3$, be a~convex body, let $\mathcal F_1,\mathcal F_2,\ldots,\mathcal F_m\subseteq\Graff_k(\mathbb R^n)$ be families of affine subspaces and let $A_1,A_2,\ldots,A_m\leq\O(k,\mathbb R)$ be closed subgroups of the orthogonal group. If all the intersections $K\cap H$ of $K$ with affine subspaces $H\in\mathcal F_i$ are invariant under the action of a~group affinely conjugate to $A_i$, $i=1,2,\ldots,m$, then $K$ is itself invariant under the actions of a~group affinely conjugate to a~subgroup $G\leq\O(n,\mathbb R)$.
\end{problem}

\subsection{Representations of the orthogonal group}\label{sec:14}

Since in high dimensions, the familiar low-dimensional concepts may become slightly awkward, let us now give the general definitions and recall some basic facts from the representation theory of Lie groups.\\

Let $K\subset\mathbb R^n$, $n\geq 3$, be a~convex body. The set $G$ of affine symmetries of $K$ (i.e., the set of affine automorphisms $g\in\GA(\mathbb R^n)$ of the ambient space such that $gK=K$) has the structure of a~compact subgroup of $\GL(n,\mathbb R)$, by a~simple geometric continuity argument. As such, it is affinely conjugate to a~subgroup of the orthogonal group $\O(n,\mathbb R)$ (see \cite[Theorem~4.38]{kirillov2017introduction}). Indeed, all affine symmetries of $K$ must also preserve, in particular, the John ellipsoid of $K$ \cite[\S4]{john1948extremum}, whose symmetry group is affinely conjugate to $\O(n,\mathbb R)$. Therefore, from now on, we will restrict our attention to the closed subgroups of the orthogonal group.\\

A single abstract Lie group may act on $\mathbb R^n$ in many essentially different ways. The simplest example is $\O(1,\mathbb R)$, the cyclic group of order $2$, which may act on $\mathbb R^4$ by reflections through subspaces of different dimensions:
$$\begin{minipage}{.25\linewidth}\begin{equation}\label{eq:01}\begin{bmatrix}\pm\\&+\\&&+\\&&&+\end{bmatrix},\end{equation}\end{minipage}
\begin{minipage}{.25\linewidth}\begin{equation}\label{eq:02}\begin{bmatrix}\pm\\&\pm\\&&+\\&&&+\end{bmatrix},\end{equation}\end{minipage}
\begin{minipage}{.25\linewidth}\begin{equation}\label{eq:03}\begin{bmatrix}\pm\\&\pm\\&&\pm\\&&&+\end{bmatrix},\end{equation}\end{minipage}
\begin{minipage}{.25\linewidth}\begin{equation}\label{eq:04}\begin{bmatrix}\pm\\&\pm\\&&\pm\\&&&\pm\end{bmatrix}.\end{equation}\end{minipage}$$
Note that \eqref{eq:01} corresponds to the group of reflections through the hyperplane $\langle\boldsymbol e_2,\boldsymbol e_3,\boldsymbol e_4\rangle$, \eqref{eq:03} corresponds to the group of reflections through the axis $\langle\boldsymbol e_4\rangle$, and \eqref{eq:04} corresponds to the group of reflections through the point $\mathbf 0$. Similarly, the special orthogonal group $\SO(3,\mathbb R)$ may act on $\mathbb R^4$ by two non-isomorphic representations (see \cite[\S6.1.2]{gl4}):
$$\begin{minipage}{.5\linewidth}\begin{equation}\label{eq:05}\begin{bmatrix}a_{11}&a_{12}&a_{13}&\\a_{21}&a_{22}&a_{23}&\\a_{31}&a_{32}&a_{33}&\\&&&1\end{bmatrix}\quad\begin{aligned}\boldsymbol A^\top\boldsymbol A&=\mathbf I_3\\\det\boldsymbol A&=1\end{aligned}\ ,\end{equation}\end{minipage}
\begin{minipage}{.5\linewidth}\begin{equation}\label{eq:06}\begin{bmatrix}a_1&-a_2&-a_3&-a_4\\a_2&a_1&-a_4&a_3\\a_3&a_4&a_1&-a_2\\a_4&-a_3&a_2&a_1\end{bmatrix}\quad\begin{aligned}\boldsymbol a^\top\boldsymbol a&=1\\\end{aligned}\ .\end{equation}\end{minipage}$$
Note that \eqref{eq:05} corresponds to the group of rotations about the axis $\langle\boldsymbol e_4\rangle$, whereas \eqref{eq:06} corresponds to the group of unit quaternions and does not have any non-trivial fixed points. Nevertheless, both \eqref{eq:05} and \eqref{eq:06} are subgroups of $\SO(4,\mathbb R)$ isomorphic to $\SO(3,\mathbb R)$. Therefore, by the action of a~group, we always mean the action of its particular representation.

\begin{definition}[{\cite[Definition~2.16, Definition~4.9]{kirillov2017introduction}}]
A \emph{representation} of a~Lie group $G$ is a~vector space $V$ together with a~group morphism $\rho:G\to\End(V)$. In other words, we assign to every $g\in G$ a~linear map $\rho(g):V\to V$ so that $\rho(g)\rho(h)=\rho(gh)$. A~\emph{subrepresentation} is a~vector subspace $W\subset V$ stable under the action, i.e., $\rho(g)W\subset W$ for all $g\in G$. If $W\subset V$ is a~subrepresentation, then the quotient space $V/W$ has a~canonical structure of a~representation. It will be called the \emph{quotient representation}.
\end{definition}

\begin{definition}\thlabel{def:04}
By the canonical action of the orthogonal group $\O(1,\mathbb R)$ by reflections through a~subspace of dimension $k$, $0\leq k<n$, we mean the representation conjugate to \eqref{eq:07} and denote it by $\pi_k$. The fixed-point subspace of this representation is called the \emph{hyperaxis of $k$-reflection} and the complementary subrepresentation space is called the \emph{hyperplane of $k$-reflection}. Similarly, by canonical action of the orthogonal group $\O(n-k,\mathbb R)$ by rotations about a~subspace of dimension $k$, $0\leq k<n$, we mean the representation conjugate to \eqref{eq:08} and denote it by $\rho_k$. The fixed-point subspace of this representation is called the \emph{hyperaxis of $k$-revolution} and the complementary subrepresentation space is called the \emph{hyperplane of $k$-revolution}.
$$\begin{minipage}{.5\linewidth}\begin{equation}\label{eq:07}\begin{tikzpicture}[baseline={([yshift=-.5\ht\strutbox+.5\dp\strutbox]base)}]
\matrix[matrix of math nodes, nodes in empty cells, nodes = {inner sep = 0pt, anchor = center}, column 1/.style = {minimum width = 7em}, column 2/.style = {minimum width = 3em}, row 1/.style = {minimum height = 7em}, row 2/.style = {minimum height = 3em}, left delimiter={[}, right delimiter={]}, row sep = \pgflinewidth, column sep = \pgflinewidth] (mat) {\mathbf I_{n-k}\otimes\O(\mathbb R^1)&\mathbf 0\\\mathbf 0&\mathbf I_k\\};
\coordinate (base) at ($(mat-1-1.north west)!.5!(mat-2-1.south west)$);
\node [anchor = base east, xshift = -2\tabcolsep, yshift=-.5\ht\strutbox+.5\dp\strutbox] at (base) {$\pi_k\sim$};
\draw [decorate, decoration = {calligraphic brace, amplitude = 5pt, raise = 3pt}, line width = 1pt] (mat-1-1.north west) -- (mat-1-1.north east) node[above = 8pt, midway] {$n-k$};
\draw [decorate, decoration = {calligraphic brace, amplitude = 5pt, raise = 3pt}, line width = 1pt] (mat-1-2.north west) -- (mat-1-2.north east) node[above = 8pt, midway] {$k$};
\draw [dashed] (mat-1-1.north west) rectangle (mat-1-1.south east);
\draw [dashed] (mat-2-2.north west) rectangle (mat-2-2.south east);
\end{tikzpicture},\end{equation}\end{minipage}
\begin{minipage}{.5\linewidth}\begin{equation}\label{eq:08}\begin{tikzpicture}[baseline={([yshift=-.5\ht\strutbox+.5\dp\strutbox]base)}]
\matrix[matrix of math nodes, nodes in empty cells, nodes = {inner sep = 0pt, anchor = center}, column 1/.style = {minimum width = 7em}, column 2/.style = {minimum width = 3em}, row 1/.style = {minimum height = 7em}, row 2/.style = {minimum height = 3em}, left delimiter={[}, right delimiter={]}, row sep = \pgflinewidth, column sep = \pgflinewidth] (mat) {\mathbf I_1\otimes\O(\mathbb R^{n-k})&\mathbf 0\\\mathbf 0&\mathbf I_k\\};
\node [anchor = base east, xshift = -2\tabcolsep, yshift=-.5\ht\strutbox+.5\dp\strutbox] at (base) {$\rho_k\sim$};
\draw [decorate, decoration = {calligraphic brace, amplitude = 5pt, raise = 3pt}, line width = 1pt] (mat-1-1.north west) -- (mat-1-1.north east) node[above = 8pt, midway] {$n-k$};
\draw [decorate, decoration = {calligraphic brace, amplitude = 5pt, raise = 3pt}, line width = 1pt] (mat-1-2.north west) -- (mat-1-2.north east) node[above = 8pt, midway] {$k$};
\draw [dashed] (mat-1-1.north west) rectangle (mat-1-1.south east);
\draw [dashed] (mat-2-2.north west) rectangle (mat-2-2.south east);
\end{tikzpicture}.\end{equation}\end{minipage}$$
\end{definition}

\subsection{Pseudo- and quasi-center of symmetry}\label{sec:15}

Finally, to phrase all the problems from different fields of mathematics using one unified language, let us introduce some important notions:

\begin{definition}
Let $K\subset\mathbb R^n$, $n\geq 3$, be a~convex body. We say that $K$ is a~body of (affine) $k$-reflection, $0\leq k<n$, if $K$ is invariant under the action of a~group (affinely) conjugate to $\pi_k$. Similarly, we say that $K$ is a~body of (affine) $k$-revolution, $0\leq k<n$, if $K$ is invariant under the action of a~group (affinely) conjugate to $\rho_k$.
\end{definition}

\begin{definition}\thlabel{def:01}
Let $K\subset\mathbb R^n$, $n\geq 3$, be a~convex body, and let $p\in\mathbb R^n$ be any point of the ambient space. We say that $p$ is a~\emph{pseudo-center of (affine) $k$-reflection}, $0\leq k<n-1$, if every intersection $K\cap H$ of $K$ with an affine hyperplane $H\in\Gr_{n-1}(\mathbb R^n)+p$ is invariant under the action of a~group (affinely) conjugate to $\pi_k$. Similarly, we say that $p$ is a~\emph{pseudo-center of (affine) $k$-revolution}, $0\leq k<n-1$, if every intersection $K\cap H$ of $K$ with an affine hyperplane $H\in\Gr_{n-1}(\mathbb R^n)+p$ is invariant under the action of a~group (affinely) conjugate to $\rho_k$.
\end{definition}

\noindent In particular, a~pseudo-center of $k$-revolution is a~pseudo-center of $k$-reflection, $0\leq k<n-1$. Note also that in \thref{def:01}, we do not assume that $p$ is fixed by the group action. Clearly, if $K$ is a~body of (affine) $k$-revolution, then every point of the ambient space is a~pseudo-center of $k$-revolution (and thus also a~pseudo-center of $k$-reflection). \thref{def:01} generalizes the concept of a~\emph{Larman point} (cf. \cite[Definition~1]{Alfonseca}) and unifies it with the previously defined concept of a~\emph{pseudo-center} (cf. \cite[Introduction]{Larman}).

\begin{definition}\thlabel{def:02}
Let $K\subset\mathbb R^n$, $n\geq 3$, be a~convex body, and let $p\in\mathbb R^n$ be any point of the ambient space. We say that $p$ is a~\emph{quasi-center of (affine) $k$-reflection}, $0\leq k<n-1$, if every intersection $K\cap H$ of $K$ with an affine hyperplane $H\in\Gr_{n-1}(\mathbb R^n)+p$ is invariant under the action of a~group (affinely) conjugate to $\pi_k$ and fixing $p$. Similarly, we say that $p$ is a~\emph{quasi-center of (affine) $k$-revolution}, $0\leq k<n-1$, if every intersection $K\cap H$ of $K$ with an affine hyperplane $H\in\Graff_{n-1}(\mathbb R^n)$ is invariant under the action of a~group (affinely) conjugate to $\rho_k$ and fixing $p$.
\end{definition}

\noindent Again, a~quasi-center of (affine) $k$-revolution is a~quasi-center of (affine) $k$-reflection, $0\leq k<n-1$. Clearly, if $K$ is a~body of (affine) $k$-revolution, then every point of the hyperaxis of $k$-revolution is a~quasi-center of $k$-revolution (and thus also a~quasi-center of $k$-reflection). \thref{def:02} generalizes the concept of a~\emph{point of revolution} (cf. \cite[Introduction]{faor} and \cite[Definition~2]{Alfonseca}). The aforementioned observation motivated the original definition, as well as the definition of a~\emph{false axis of revolution} (cf. \cite[Introduction]{faor}).

\begin{notation}
Throughout the paper, we understand the prefix $k$- modulo the dimension of the underlying space, e.g., by $(-k)$- we will denote the action of a~group with the fixed-point subspace of codimension $k$.
\end{notation}

\section{Historical background and known results}\label{sec:01}

In this section, we will give a~profound context for our results and formulate some related theorems for later reference.

\subsection{Point reflection}\label{sec:08}

To the best of our knowledge, the whole topic was initiated by H.~Brunn (1889), who in his habilitation thesis proved the following seminal theorem:

\begin{theorem}[{\cite[IV.4]{brunn1889curven}}]\thlabel{thm:03}
Let $K\subset\mathbb R^n$, $n\geq 3$, be a~strictly convex body. If all the intersections $K\cap H$ of $K$ with affine hyperplanes $H\in\Graff_{n-1}(\mathbb R^n)$ are centrally symmetric, then $K$ is an ellipsoid.
\end{theorem}

\noindent Actually, Brunn proved \thref{thm:03} only under an additional implicit regularity assumption, which was relaxed much later by G.R.~Burton (1977) \cite[Theorem~2, Theorem~4]{Burton1977}. But even though he assumed that only the sections of sufficiently small diameter are centrally symmetric, it was still an open subset of all the sections of $K$.\\

Meanwhile, W.~Blaschke and G.~Hessenberg (1918) proved an analogous theorem about projections in a~$3$-dimensional space \cite[(II)]{Blaschke1918}. The result was further developed by C.A.~Rogers (1965), who generalized it to arbitrary dimensions. Unable to apply the theory of polar reciprocal convex bodies, he also gave a~direct proof of the dual theorem:

\begin{theorem}[{\cite[Theorem~4 and remark below]{Rogers1965}}]\thlabel{thm:01}
Let $K\subset\mathbb R^n$, $n\geq 3$, be a~convex body and let $p\in\Int K$ be an interior point of $K$. If all the intersections $K\cap H$ of $K$ with affine hyperplanes $H\in\Gr_k(\mathbb R^n)+p$, $2\leq k\leq n$, are centrally symmetric, then $K$ itself is centrally symmetric.
\end{theorem}

\noindent Note that in \thref{thm:01}, we neither assume that $p$ is the center of symmetry for all the sections nor conclude that $p$ is the center of symmetry for $K$. Actually, Rogers called point $p$ a~\emph{false center of symmetry} if $p$ is a~pseudo-center but not a~center of symmetry, and conjectured that the only $n$-dimensional convex bodies with a~false center of symmetry are ellipsoids:

\begin{theorem}[{False Center Theorem \cite[Conjecture]{Rogers1965}}]\thlabel{thm:11}
Let $K\subset\mathbb R^n$, $n\geq 3$, be a~convex body, and let $p\in\mathbb R^n$ be any point of the ambient space. If all the intersections $K\cap H$ of $K$ with affine hyperplanes $H\in\Gr_{n-1}(\mathbb R^n)+p$ are centrally symmetric, then either $K$ is centrally symmetric with respect to the point $p$ or $K$ is an ellipsoid.
\end{theorem}

\noindent A~deep result of P.W.~Aitchison, C.M.~Petty and C.A.~Rogers (1971) \cite[Theorem~1, Theorem~2]{Aitchison_Petty_Rogers_1971} asserts the positive answer, provided that $p$ is an interior point of $K$. In its full generality, the conjecture was finally proved by D.G.~Larman (1974) \cite[Theorem]{Larman}. Much later, L.~Montejano and E.~Morales-Amaya (2007+) gave several new, simpler proofs \cite[Remark~1, Remark~2, Remark~3]{Montejano_Morales-Amaya_2007} \cite[Theorem~1.1 and remark at the end]{montejano2021new}.\\

Independently, S.P.~Olovyanishnikov (1941) developed \thref{thm:03} in a~completely different direction, proving that it is enough to assume that only the sections that divide the volume of $K$ in a~fixed ratio $0<\lambda<1$ are centrally symmetric:

\begin{theorem}[{\cite[Teorema and subsequent Obobshcheniye]{Olov}}]\thlabel{thm:12}
Let $K\subset\mathbb R^n$, $n\geq 3$, be a~convex body. If all the intersections $K\cap H$ of $K$ with affine hyperplanes $H$ that divide the volume of $K$ in a~fixed ratio $0<\lambda<1$ are centrally symmetric, then $K$ is an ellipsoid.
\end{theorem}

\noindent However, such a~rigid geometric constraint seems to be far from necessary (cf. \cite[Conjecture~1]{BARKER200179}), especially since L.~Montejano and E.~Morales-Amaya (2007) proved much later the following:

\begin{theorem}[{Shaken False Center Theorem \cite[Theorem~1]{S002557930000019X}}]\thlabel{thm:04}
Let $K\subset\mathbb R^3$ be an origin-symmetric convex body and let $\delta:\mathbb S^2\to\mathbb R$ be an even continuous map which is $C^1$ in a~neighbourhood of $\delta^{-1}(\{0\})$. Denote $H_{\boldsymbol\xi}^\delta\colonequals\langle\boldsymbol\xi\rangle^\perp+\delta(\boldsymbol\xi)\boldsymbol\xi$, $\boldsymbol\xi\in\mathbb S^2$. If for every $\boldsymbol\xi\in\mathbb S^2$, either $\delta(\boldsymbol\xi)=0$ and $K\cap H_{\boldsymbol\xi}^\delta$ is an ellipse or $\delta(\boldsymbol\xi)\neq 0$ and $K\cap H_{\boldsymbol\xi}^\delta$ is centrally symmetric, then $K$ is an ellipsoid.
\end{theorem}

\noindent Although at first glance \thref{thm:04} might seem to be the most general statement that we may hope for, there are still some issues that need to be resolved. First, the result is known only in the $3$-dimensional space. Secondly, the convex body $K$ is \emph{a priori} assumed to be centrally symmetric. L.~Montejano (2004) proved that under some other condition on the function $\delta$, the latter is necessarily the case:

\begin{theorem}[{\cite[Theorem~1.2]{MR2168169}}]\thlabel{thm:05}
Let $K\subset\mathbb R^3$ be a~strictly convex body and let $\delta:\mathbb S^2\to\mathbb R$ be an odd continuous map such that $H_{\boldsymbol\xi_1}^\delta\cap H_{\boldsymbol\xi_2}^\delta\cap\Int K\neq\emptyset$ for every $\boldsymbol\xi_1,\boldsymbol\xi_2\in\mathbb S^2$. If for every $\boldsymbol\xi\in\mathbb S^2$, $K\cap H_{\boldsymbol\xi}^\delta$ is centrally symmetric, then $K$ itself is centrally symmetric.
\end{theorem}

\noindent Unfortunately, constraints arising from \thref{thm:04,thm:05} exhibit completely different topological properties. Trying to reconcile them may prove difficult, as the non-example \cite[Introduction]{S002557930000019X} given by Montejano and Morales-Amaya seems to indicate. Therefore, in our conjectures, we dare not ask about the shaken variant, although we believe that, when formulated properly, it should generally hold (cf. \cite[Conjecture~1]{Bianchi1987}). A~strong argument that this may indeed be the case was provided by E.~Morales-Amaya (2023), who recently proved the following:

\begin{theorem}[{\cite[Theorem~1]{moralesamaya2023}}]
Let $K\subset\mathbb R^n$, $n\geq 3$, be an origin-symmetric convex body and let $B\subset\Int K$ be a~Euclidean ball that does not contain the origin. If all the intersections $K\cap H$ of $K$ with affine hyperplanes $H$ tangent to $B$ are centrally symmetric, then $K$ is an ellipsoid.
\end{theorem}

For a~much more detailed account, we refer the reader to the expository paper \cite[\S4]{Soltan2019}.

\subsection{Affine involutions}\label{sec:03}

Surprisingly little is known about affine involutions other than reflection through a~point. The main driving force for studying this case was the celebrated conjecture of K.~Bezdek (1999):

\begin{conjecture}[{cf. \cite[\S1.4]{Odor1999}}]\thlabel{con:02}
Let $K\subset\mathbb R^3$ be a~convex body. If all the intersections $K\cap H$ of $K$ with affine planes $H\in\Graff_2(\mathbb R^3)$ are bodies of $1$-reflection, then $K$ is either a~body of $1$-revolution or an ellipsoid.
\end{conjecture}

\noindent Since an ellipsoid is a~body of \emph{affine} $1$-revolution, but in general not a~body of $1$-revolution, it begs to consider also the affine counterpart of the conjecture:

\begin{conjecture}[{cf. \cite[\S1.4]{Odor1999}}]\thlabel{con:03}
Let $K\subset\mathbb R^3$ be a~convex body. If all the intersections $K\cap H$ of $K$ with affine planes $H\in\Graff_2(\mathbb R^3)$ are bodies of affine $1$-reflection, then $K$ is a~body of affine $1$-revolution.
\end{conjecture}

\noindent In these problems, it is not enough to consider only planes passing through a~fixed point, even if this point is assumed to be a~quasi-center of $1$-reflection. Indeed, L.~Montejano remarked in \cite[\S2]{MR2168169} that if $K$ is the convex hull of two different and concentric unit disks\footnote{i.e., the set defined by inequality $(|x|+|y|)^2+z^2\leq 1$}, then their common center is a~quasi-center of $1$-reflection for $K$.

\begin{remark}
Although at first glance \thref{con:03} seems more general than \thref{con:02}, there is no trivial implication in either direction. Applying \thref{con:03} to the hypothesis of \thref{con:02} yields merely that $K$ is a~body of \emph{affine} $1$-revolution, but \emph{a priori} not a~body of $1$-revolution. Besides, one has to somehow capture the dichotomy between the cases of a~body of $1$-revolution and of an ellipsoid, which does not arise in the affine setting.
\end{remark}

T.~\'Odor (1999) claimed to have confirmed \thref{con:02}, but unfortunately, his approach was found incomplete. To the best of our knowledge, both problems remain open. However, there are several related results that certainly represent some progress towards the proof. In his remarkable paper, L.~Montejano (2004) showed that if a~convex body $K$ admits a~pseudo-center of $1$-reflection $p$ in its interior, then there is a~plane $H\in\Gr_2(\mathbb R^3)+p$ such that $K\cap H$ is a~disk \cite[Theorem~2.1]{MR2168169}. Later, J.~Jer\'onimo-Castro, L.~Montejano and E.~Morales-Amaya showed that if every point on some straight line $L\in\Graff_1(\mathbb R^3)$ is a~quasi-center of $1$-reflection for a~strictly convex body $K$, then it is either a~body of $1$-revolution or an ellipsoid \cite[Theorem~1.3]{faor}. However, this is not yet the proof of \thref{con:02}, since we have to assume that all the planar sections of $K$ admit an axis of $1$-reflection intersecting some fixed straight line $L$. Recently, M.~Angeles Alfonseca, M.~Cordier, J.~Jer\'onimo-Castro, and E.~Morales-Amaya (2024) improved that result for centrally symmetric $K$, by assuming that it admits merely two quasi-centers of $1$-reflection \cite[Corollary~1]{Alfonseca}. They also contributed the following theorem:

\begin{theorem}[{cf. \cite[Theorem~1]{Alfonseca}}]\thlabel{thm:06}
Let $K\subset\mathbb R^n$ be a~strictly convex body centrally symmetric with respect to the origin and let $p\in\mathbb R^n\setminus\{\mathbf 0\}$ be any point of the ambient space, different from the origin. If $p$ is a~quasi-center of $1$-reflection for $K$, then $K$ is a~body of $1$-revolution, with axis of $1$-revolution passing through $\mathbf 0$ and $p$.
\end{theorem}

\noindent Although the authors formulated it differently, their argument proves \thref{thm:06} only as stated above, and not as stated in \cite[Theorem~1]{Alfonseca}. Both claims are equivalent in dimension $3$, but in higher dimensions, the original one does not hold. The seemingly superfluous central symmetry assumption plays an important role in the proof and can not be easily relaxed.\\

In the same paper \cite{Alfonseca}, the authors proposed several ways of generalizing Bezdek's conjecture to higher dimensions, by assuming that all the intersections $K\cap H$ of $K$ with affine hyperplanes $H\in\Graff_{n-1}(\mathbb R^n)$ are bodies of either \begin{enumerate*}[label=(\alph*)]\item $1$-reflection; or\item $(-1)$-reflection being the same as $(-1)$-revolution; or\item $1$-revolution.\end{enumerate*} Note that when $n=3$, all three conditions coincide and reduce to the original Bezdek's condition. There is an even more general question we would like to see resolved:

\begin{question}[{cf. \cite[Conjecture~2]{Alfonseca}}]\thlabel{con:12}
Let $K\subset\mathbb R^n$, $n\geq 3$, be a~convex body, and let $p\in\mathbb R^n$ be any point of the ambient space. If $p$ is a~pseudo-center of (affine) $k$-reflection for $K$, $0\leq k<n-1$, but not the center of symmetry for $K$, is $K$ itself a~body of (affine) $k$-revolution or an ellipsoid?
\end{question}

\noindent If we were able to prove the affine variant of \thref{con:12}, the orthogonal variant would follow immediately from a~general \thref{lem:01} in this paper. The assumption that $p$ is not the center of symmetry for $K$ is essential (see \cite[\S2]{MR2168169}). We believe that this is the only obstacle on the way to relaxing the hypothesis of Bezdek's conjecture by considering only hyperplanes passing through a~fixed point. Nevertheless, we also propose the following problem, which may be solvable using local methods:

\begin{question}[{cf. \cite[\S1.4]{Odor1999}}]\thlabel{con:04}
Let $K\subset\mathbb R^n$, $n\geq 3$, be a~convex body. If all the intersections $K\cap H$ of $K$ with affine hyperplanes $H\in\Graff_{n-1}(\mathbb R^n)$ are bodies of (affine) $k$-reflection, $0\leq k<n-1$, but not the center of symmetry for $K$, is $K$ a~body of (affine) $k$-revolution or an ellipsoid?
\end{question}

\noindent Recall that if $K$ is a~body of (affine) $k$-revolution, then every hyperplanar section of $K$ is a~body of (affine) $k$-reflection. On the other hand, if $K$ is simply a~body of (affine) $k$-reflection, then only its intersections with affine hyperplanes at specific directions are \emph{a priori} known to possess any symmetry. Since in \thref{con:04}, this case eludes us, we would like to pose a~different, even more interesting question:

\begin{question}
Let $K\subset\mathbb R^n$, $n\geq 4$, be a~convex body. If for every affine hyperplane $H\in\Graff_{n-1}(\mathbb R^n)$ there exists an affine hyperplane $H'\in\Graff_{n-2}(H)$ such that the intersection $K\cap H'$ is a~body of (affine) $(k-1)$-reflection, $1\leq k<n-1$, is $K$ itself a~body of (affine) $k$-reflection?
\end{question}

Finally, a~remarkable result, though of a~slightly different flavor, was obtained by S.~Myroshnychenko, D.~Ryabogin, and C.~Saroglou (2017):

\begin{theorem}[{\cite[Corollary~1.7]{10.1093/imrn/rnx211}}]\thlabel{thm:10}
Let $K\subset\mathbb R^n$, $n\geq 3$, be a~convex body, and let $p\in\mathbb R^n$ be any point of the ambient space. If all the intersections $K\cap H$ of $K$ with affine hyperplanes $H\in\Gr_{n-1}(\mathbb R^n)+p$ are invariant under the action of the hyperoctahedral group $\B_{n-1}$, then $K$ is a~Euclidean ball.
\end{theorem}

\noindent Their argument hangs on the fact that the only ellipsoid invariant under the action of the hyperoctahedral group is a~Euclidean ball. In other words, the hyperoctahedral group has a~sufficiently rich structure. So a~natural question arises:

\begin{question}[{cf. \cite[Question~4.8]{10.1093/imrn/rnx211}}]\thlabel{con:08}
Let $K\subset\mathbb R^n$, $n\geq 3$, be a~convex body, and let $p\in\mathbb R^n$ be any point of the ambient space. If all the intersections $K\cap H$ of $K$ with affine hyperplanes $H\in\Gr_{n-1}(\mathbb R^n)+p$ are invariant under the action of (a group affinely conjugate to) the hyperoctahedral group $\B_{n-1-k}$, $0\leq k<n-2$, is $K$ a~body of (affine) $k$-revolution?
\end{question}

\noindent A~positive answer to \thref{con:08} would imply a~positive answer to \thref{con:05} in a~later section of this paper.

\subsection{Rotations about a~subspace}\label{sec:16}

The problem becomes substantially easier if we assume that the hyperplanar sections of $K$ are invariant under the action of a~group of positive dimension. In the simplest case, when the group is affinely conjugate to the entire orthogonal group, it reduces to the classical question of characterizing convex quadrics by their hyperplanar sections. In such a~disguise, the problem was first considered by T.~Kubota (1914), who proved the following characterization theorems:

\begin{theorem}[{\cite[\S III]{Kubota}}]\thlabel{thm:17}
Let $K\subset\mathbb R^3$ be a~convex body and let $p\in\mathbb R^3$ be any point of the ambient space. If all the intersections $K\cap H$ of $K$ with affine planes $H\in\Gr_2(\mathbb R^3)+p$ are ellipses, then $K$ itself is an ellipsoid.
\end{theorem}

\begin{theorem}[{\cite[\S V]{Kubota}}]\thlabel{thm:07}
Let $K\subset\mathbb R^3$ be a~convex body and let $L\subsetneq K$ be a~convex body properly contained in $K$. If all the intersections $K\cap H$ of $K$ with affine planes $H$ supporting $L$ are ellipses, then $K$ itself is an ellipsoid.
\end{theorem}

\noindent A~similar method was independently used much later by H.~Busemann (1955), who proved \thref{thm:17} in arbitrary dimension:

\begin{theorem}[{\cite[(16.12)]{busemann1955geometry}}]\thlabel{thm:23}
Let $K\subset\mathbb R^n$, $n\geq 3$, be a~convex body, and let $p\in\mathbb R^n$ be any point of the ambient space. If all the intersections $K\cap H$ of $K$ with affine hyperplanes $H\in\Gr_k(\mathbb R^n)+p$, $2\leq k\leq n$, are ellipsoids, then $K$ itself is an ellipsoid.
\end{theorem}

\noindent Note a~striking resemblance of \thref{thm:23} to \thref{thm:01} and \thref{thm:07} to \thref{thm:12}. Finally, G.~Bianchi and P.M.~Gruber (1987) gave a~far-reaching generalization of Busemann's results, thus completely solving the problem:

\begin{theorem}[{\cite[Theorem~1]{Bianchi1987}}]\thlabel{thm:20}
Let $K\subset\mathbb R^n$, $n\geq 3$, be a~convex body and let $\delta:\mathbb S^{n-1}\to\mathbb R$ be a~continuous map. If for every $\boldsymbol\xi\in\mathbb S^{n-1}$, $K\cap H_{\boldsymbol\xi}^\delta$ is an ellipsoid, then $K$ itself is an ellipsoid.
\end{theorem}

\noindent For a~much more detailed account, we refer the reader to the expository paper \cite[\S3]{Soltan2019}.\\

Amazingly, representations other than $\rho_0$ were not considered at all, until the remarkable paper on the isometric conjecture of S.~Banach by G.~Bor, L.~Hern\'andez-Lamoneda, V.~Jim\'enez de Santiago and L.~Montejano (2021) \cite{10.2140/gt.2021.25.2621} (again, we refer the reader to the expository paper \cite[\S6]{Soltan2019} for an overview of this otherwise interesting problem). Namely, they asked the following question:

\begin{question}[{\cite[Remark~2.9]{10.2140/gt.2021.25.2621}}]\thlabel{con:06}
Let $K\subset\mathbb R^n$, $n\geq 4$, be a~convex body and let $p\in\Int K$ be an interior point of $K$. If $p$ is a~pseudo-center of $1$-revolution for $K$, is $K$ itself a~body of $1$-revolution?
\end{question}

\noindent B.~Zawalski (2023) confirmed \thref{con:06} under additional assumption that $p$ is a~quasi-center of $1$-revolution and the boundary of $K$ is sufficiently smooth:

\begin{theorem}[{cf. \cite[Theorem~1.2]{Zawalski2024}}]\thlabel{thm:13}
Let $K\subset\mathbb R^n$, $n\geq 4$, be a~convex body with boundary of differentiability class $C^3$ and let $p\in\mathbb R^n$ be any point of the ambient space. If $p$ is a~quasi-center of affine $1$-revolution for $K$, then $K$ itself is a~body of affine $1$-revolution.
\end{theorem}

\noindent For the orthogonal variant of \thref{thm:13}, see also \thref{thm:15} in this paper. Originally, Zawalski assumed that $p$ is the center of symmetry for $K$, but in the course of the proof, he used only the fact that the axis of $1$-revolution of every hyperplanar section passes through $p$. Admittedly, at that time, the author was not aware of the notion of a~quasi-center, already established in the literature. It should be noted that although his argument uses several facts from \cite{10.2140/gt.2021.25.2621} that were originally formulated only for symmetric convex bodies (i.e., unit balls in Banach spaces), all of them are true as well for general convex bodies, in most cases with the same proof. Under the additional assumption that $K$ has a~unique diameter $D$ and $p\notin D$, \thref{con:06} was independently confirmed by M.~Angeles Alfonseca, M.~Cordier, J.~Jer\'onimo-Castro and E.~Morales-Amaya (2024) \cite[Theorem~5]{Alfonseca}.\\

Let us conclude with a~general question we would like to see resolved:

\begin{question}\thlabel{con:05}
Let $K\subset\mathbb R^n$, $n\geq 4$, be a~convex body, and let $p\in\mathbb R^n$ be any point of the ambient space. If $p$ is a~pseudo-center of (affine) $k$-revolution for $K$, $0\leq k<n-1$, is $K$ itself a~body of (affine) $k$-revolution?
\end{question}

\noindent As in the case of \thref{con:12}, if were able to prove the affine variant of \thref{con:05}, the orthogonal variant would follow immediately from general \thref{lem:01,lem:07} in this paper.

\section{An aligned center of symmetry}\label{sec:02}

Our ultimate goal is to characterize high-dimensional bodies of (affine) $k$-revolution in the spirit of Roger's \thref{thm:01}, through a~suitable generalization of Bezdek's conjecture (i.e., \thref{con:02,con:03}). However, since the assumption of symmetry only for hyperplane sections passing through a~fixed point is known to be insufficient (see \cite[\S2]{MR2168169}), we inevitably need to impose some additional conditions on the structure of the family of symmetric sections. If $K$ is a~body of (affine) $k$-revolution, then every point $p$ of the hyperaxis of revolution is a~quasi-center of (affine) $k$-revolution for $K$, whereas every point $p$ of the complement of the hyperaxis of revolution is a~pseudo-center of (affine) $k$-revolution, but not a~quasi-center of (affine) $k$-revolution for $K$, unless $K$ is an ellipsoid. In light of \cite[Conjecture~2]{Alfonseca}, the assumption that $p$ is of the second type should be enough to conclude that $K$ is a~body of (affine) $k$-revolution. In this paper, however, we will focus on the points of the first type.\\

Before attacking an open problem, it is usually worth considering some special cases first. J.~Jer\'onimo-Castro, L.~Montejano, and E.~Morales-Amaya had the promising idea of taking as a~quasi-center of $1$-revolution a~regular point on the boundary:

\begin{theorem}[{\cite[Lemma~3.2]{faor}}]\thlabel{thm:08}
Let $K\subset\mathbb R^3$ be a~convex body and let $p\in\partial K$ be any smooth point of the boundary. If $p$ is a~quasi-center of $1$-revolution for $K$, then $K$ is a~body of $1$-revolution, with the axis of $1$-revolution passing through $p$ and perpendicular to $\partial K$.
\end{theorem}

\noindent Their simple proof hangs on the observation that if a~regular point $p$ on the boundary of $K\cap H$ is a~fixed point of some affine symmetry $S$ of $K\cap H$, then the affine hyperplane tangent to $K\cap H$ at $p$ is an invariant subspace of $S$. Let us isolate this property:

\begin{definition}\thlabel{def:03}
Let $K\subset\mathbb R^n$, $n\geq 3$, be a~convex body, and let $p\in\mathbb R^n$ be any point of the ambient space. We say that $p$ is an \emph{$m$-aligned pseudo-center of (affine) $k$-reflection}, $n-k\leq m<n$, $1\leq k<n-1$, if there exists an $m$-dimensional affine subspace $T\in\Gr_m(\mathbb R^n)+p$ such that every intersection $K\cap H$ of $K$ with an affine hyperplane $H\in\Gr_{n-1}(\mathbb R^n)+p$ is invariant under the action of a~group (affinely) conjugate to $\pi_k$, with hyperplane of reflection contained in $T\cap H$. Similarly, we say that $p$ is an \emph{$m$-aligned pseudo-center of (affine) $k$-revolution}, $n-k\leq m<n$, $1\leq k<n-1$, if there exists an $m$-dimensional affine subspace $T\in\Gr_m(\mathbb R^n)+p$ such that every intersection $K\cap H$ of $K$ with an affine hyperplane $H\in\Gr_{n-1}(\mathbb R^n)+p$ is invariant under the action of a~group (affinely) conjugate to $\rho_k$, with hyperplane of revolution contained in $T\cap H$.
\end{definition}

\begin{definition}\thlabel{def:05}
Let $K\subset\mathbb R^n$, $n\geq 3$, be a~convex body, and let $p\in\mathbb R^n$ be any point of the ambient space. We say that $p$ is an \emph{$m$-aligned quasi-center of (affine) $k$-reflection}, $n-k\leq m<n$, $1\leq k<n-1$, if there exists an $m$-dimensional affine subspace $T\in\Gr_m(\mathbb R^n)+p$ such that every intersection $K\cap H$ of $K$ with an affine hyperplane $H\in\Gr_{n-1}(\mathbb R^n)+p$ is invariant under the action of a~group (affinely) conjugate to $\pi_k$ and fixing $p$, with hyperplane of reflection contained in $T\cap H$. Similarly, we say that $p$ is an \emph{$m$-aligned quasi-center of (affine) $k$-revolution}, $n-k\leq m<n$, $1\leq k<n-1$, if there exists an $m$-dimensional affine subspace $T\in\Gr_m(\mathbb R^n)+p$ such that every intersection $K\cap H$ of $K$ with an affine hyperplane $H\in\Gr_{n-1}(\mathbb R^n)+p$ is invariant under the action of a~group (affinely) conjugate to $\rho_k$ and fixing $p$, with hyperplane of revolution contained in $T\cap H$.
\end{definition}

\noindent Recall that we understand the prefix $k$- modulo the dimension of the underlying space.\\

Interestingly, the concept of an aligned center of symmetry has already been studied in the literature, albeit implicitly. It is easy to see that the original proof of \thref{thm:08} works equally well for any point $p\in\mathbb R^3$, not necessarily on the boundary of $K$, being a~$(-1)$-aligned quasi-center of $1$-reflection (or, equivalently, $1$-revolution). Moreover, in the proof of \cite[Theorem~1]{Alfonseca}, the assumption that $K$ is a~strictly convex body centrally symmetric with respect to the origin is required only to show that the quasi-center of $1$-reflection $p\in\mathbb R^n$ is actually $(-1)$-aligned.

\begin{remark}
In general, if a~smooth point $p\in\partial K$ is a~quasi-center of any (affine) symmetry for a~convex body $K$, then it is necessarily $(-1)$-aligned. Indeed, let $T\in\Gr_{n-1}(\mathbb R^n)+p$ be the unique hyperplane tangent to $K$ at $p$ and let $H\in\Gr_{n-1}(\mathbb R^n)+p$ be any affine hyperplane different from $T$. Since the image of the unique hyperplane $T\cap H$ tangent to $K\cap H$ at $p$ under any affine symmetry of $K\cap H$ fixing $p$ is again the same unique tangent hyperplane, $T\cap H$ is invariant under the action of the compact group of all affine symmetries of $K\cap H$ fixing $p$. Furthermore, since $K$ is contained in one of the two half-spaces determined by $T$, the subrepresentation complementary to $T\cap H$ must be orientation-preserving, and hence trivial.
\end{remark}

The property of being an aligned quasi-center is not, however, limited exclusively to the points $p$ on the boundary. If $K$ is a~body of (affine) $k$-revolution with hyperplane of (affine) revolution $T$, then every point $p$ of the hyperaxis of $k$-revolution is a~$(-k)$-aligned quasi-center of $k$-revolution (and thus also a~$(-k)$-aligned quasi-center of $k$-reflection). Indeed, for any hyperplane $H\in\Gr_{n-1}(\mathbb R^n)$, the hyperplane of (affine) revolution for $K\cap(H+p)$ is the intersection $T\cap H$. It is, therefore, natural to assume that a~quasi-center of (affine) $k$-reflection is $(-k)$-aligned. This assumption should be enough to conclude that $K$ is a~body of (affine) $k$-revolution. In the orthogonal setting, we prove it in \thref{thm:22}.\\

Now, the first question that immediately arises is whether an affine analog of \thref{thm:08} also holds. In \cite{faor}, the authors observed that the axis of revolution of $K\cap H$ (necessarily passing through $p$) must be perpendicular to the line tangent to $K\cap H$ at $p$ and hence its position is uniquely determined. Unfortunately, in the affine setting, there is no such simple argument, since the standard normal line has to be replaced with the \emph{affine normal line} (see \cite[Definition~II.3.1]{nomizu1994affine}). And this is where the new idea must come in.

\subsection{A general proof scheme}\label{sec:07}

\begin{figure}
\includegraphics{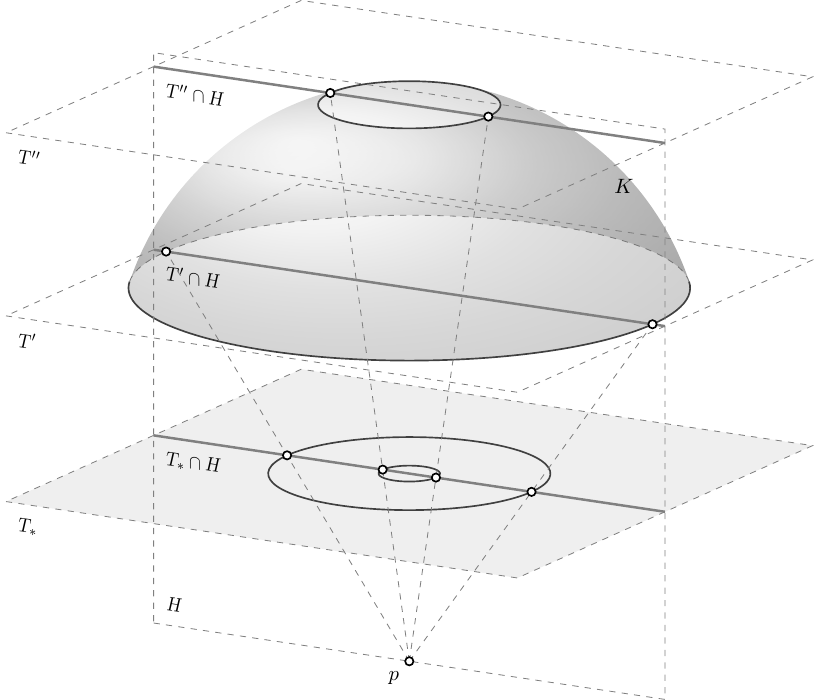}
\caption{Reducing the problem of a~$(-1)$-aligned quasi-center to a~many-body problem in codimension $1$}
\label{fig:02}
\end{figure}

That being so, suppose that $p\in\mathbb R^n$ is a~$(-1)$-aligned quasi-center of affine $k$-reflection (resp. $k$-revolution) for $K$, $1\leq k<n-1$ (see \fref{fig:02}). Let $T\in\Gr_{n-1}(\mathbb R^n)$ be the affine hyperplane from the definition and let $T'\in\Graff_{n-1}(\mathbb R^n)$ be any affine hyperplane parallel to $T$ and intersecting $K$. Also, let $H\in\Gr_{n-1}(\mathbb R^n)+p$ be any affine hyperplane. By definition, $K\cap H$ is invariant under the action of a~group $G(H)$ affinely conjugate to $\pi_k$ (resp. $\rho_k$) with invariant subspace $(T+p)\cap H$ and fixing $p$. But since the quotient representation of $G(H)$ is trivial, $T'\cap H$ is also an invariant subspace of $G(H)$. Because the choice of the affine hyperplane $H$ was arbitrary, it follows that all the intersections $K\cap T'\cap H'$ of $K\cap T'$ with affine hyperplanes $H'\in\Graff_{n-2}(T')$ are bodies of affine $(k-1)$-reflection (resp. $(k-1)$-revolution).\\

Now, even if we are able to conclude that $K\cap T'$ is a~body of affine $(k-1)$-revolution for every affine hyperplane $T'$ parallel to $T$, we still have to show that all the hyperaxes of affine revolution for different affine hyperplanes $T'$ are $1$-codimensional subspaces of a~fixed affine subspace of the ambient space, being \emph{a posteriori} the hyperaxis of affine revolution for $K$. This turns out to be non-trivial and leads to many new, interesting questions, which we discuss in \sref{sec:05}.\\

To understand the difficulty, fix any affine hyperplane $T_*$ parallel to $T$ but not passing through $p$ and denote by $\phi$ the central projection of the ambient space $\mathbb R^n\setminus(T+p)$ without the affine hyperplane $T+p$ onto the affine hyperplane $T_*$ from the point $p$ (\fref{fig:02}). Since the affine hyperplanes $T_*$ and $T'$ are parallel, the projectivity $\phi$ restricted to $T'$ is an affine isomorphism. Hence the image $\phi(K\cap T'\cap H)=\phi(K\cap T')\cap H$ is likewise invariant under the action of the same group $G(H)$ on the affine subspace $T_*\cap H$. Thus, denoting by $\boldsymbol n$ the unit normal vector of $T_*$, we have obtained a~continuous family $(K_\tau)_{\tau\in\mathbb R}$ of convex bodies
$$K_\tau\colonequals\phi(K\cap(T_*+\tau\boldsymbol n)),$$
whose hyperplanar sections $(K_\tau\cap H')_{\tau\in\mathbb R}$, for any fixed $H'\in\Graff_{n-2}(T_*)$, are all invariant under the action of the same group affinely conjugate to $\pi_{k-1}$ (resp. $\rho_{k-1}$). Clearly, we would like to conclude that all $K_\tau$, $\tau\in\mathbb R$, are bodies of affine $(k-1)$-revolution, invariant under the action of the same group affinely conjugate to $\rho_{k-1}$, but it turns out to be a~quite perplexing problem. Moreover, it seems to have not been considered in the literature so far.\\

For our purposes, in \sref{sec:05} we have made initial progress toward the ultimate solution. Specifically, we solved the problem completely in $\mathbb R^2$ (see \thref{lem:03}), and in higher dimensions, we obtained a~solution under the additional assumption that all $K_\tau$, $\tau\in\mathbb R$, are ellipsoids (see \thref{lem:05}). This latter condition can be derived independently when the symmetry group is sufficiently rich, based on other known results (i.a., \thref{thm:03,thm:20}).

\subsection{Statements of the new results}\label{sec:09}

The first of our theorems is an immediate generalization of \thref{thm:08}, both in terms of dimension of the ambient space and orthogonality of the symmetry group:

\begin{theorem}[cf. \thref{intro:01}]\thlabel{thm:14}
Let $K\subset\mathbb R^n$, $n\geq 3$, be a~convex body, and let $p\in\mathbb R^n$ be any point of the ambient space. If $p$ is a~$(-1)$-aligned quasi-center of (affine) $1$-reflection for $K$, then $K$ is a~body of (affine) $1$-revolution, with axis of $1$-revolution passing through $p$, or an ellipsoid.
\end{theorem}

\noindent An analogous theorem for a~$(-1)$-aligned quasi-center of (affine) $1$-revolution is just a~trivial corollary, as the hypothesis is strictly stronger:

\begin{corollary}[cf. \thref{intro:02}]\thlabel{thm:18}
Let $K\subset\mathbb R^n$, $n\geq 3$, be a~convex body, and let $p\in\mathbb R^n$ be any point of the ambient space. If $p$ is a~$(-1)$-aligned quasi-center of (affine) $1$-revolution for $K$, then $K$ is a~body of (affine) $1$-revolution, with axis of $1$-revolution passing through $p$.
\end{corollary}

By virtue of \thref{lem:07}, we no longer need to distinguish the case of an ellipsoid in \thref{thm:18}. Such a~dichotomy seems to be a~phenomenon occurring only for discrete groups. Indeed, every ellipsoid is a~body of $k$-reflection, but not necessarily a~body of $k$-revolution, although it is always a~body of \emph{affine} $k$-revolution, $0\leq k<n-1$. Note also that, unlike \thref{thm:15}, of which \thref{thm:18} is a~special case, our alternative proof does not require any smoothness assumption.\\

While carefully examining the proof of \thref{thm:14}, one can come to the conclusion that the hypothesis was actually excessive and a~$1$-codimensional family of affine hyperplanes passing through a~$(-1)$-aligned quasi-center should have been enough. And indeed it is so, but then some delicate issues arise that have not occurred before:

\begin{theorem}[cf. \thref{intro:03}]\thlabel{thm:16}
Let $K\subset\mathbb R^4$ be a~strictly convex body, let $p\in\Int K$ be any point in the interior of $K$, let $T\in\Gr_3(\mathbb R^4)$ be any hyperplane and let $\sigma:\Gr_2(T)\to\Gr_3(\mathbb R^4)$ be a~Lipschitz continuous map such that $\sigma(H)\cap T=H$ for all $H\in\Gr_2(T)$. If every intersection $K\cap(\sigma(H)+p)$ is invariant under the action of a~group affinely conjugate to $\pi_1$ with hyperplane of affine reflection $H$ and fixing $p$, then $K$ is a~body of affine $1$-reflection, with the axis of $1$-reflection passing through $p$.
\end{theorem}

\noindent The hyperplane $T$ plays the same role here as in the definition of an aligned center of symmetry. The map $\sigma$ should be seen as a~way of representing a~$1$-codimensional family of affine hyperplanes passing through $p$. However, it implicitly imposes some topological constraints that will eventually prove crucial. The Lipschitz continuity condition has a~deep geometric significance and is not purely technical, as it might seem (cf. \thref{thm:04,thm:05}). We intentionally do not formulate the orthogonal variant of \thref{thm:16}, as it would require a~new, different argument, which we will not give here.\\

This time, the analogous theorem for the group of rotations has a~stronger conclusion and, moreover, can be proved in arbitrary dimension:

\begin{theorem}[cf. \thref{intro:04}]\thlabel{thm:02}
Let $K\subset\mathbb R^n$, $n\geq 4$, be a~convex body, let $p\in\Int K$ be any point in the interior of $K$, let $T\in\Gr_{n-1}(\mathbb R^n)$ be any hyperplane and let $\sigma:\Gr_{n-2}(T)\to\Gr_{n-1}(\mathbb R^n)$ be a~continuous map such that $\sigma(H)\cap T=H$ for all $H\in\Gr_{n-2}(T)$. If every intersection $K\cap(\sigma(H)+p)$ is invariant under the action of a~group affinely conjugate to $\rho_1$ with hyperplane of affine revolution $H$ and fixing $p$, then $K$ is a~body of affine $1$-revolution, with the axis of $1$-revolution passing through $p$.
\end{theorem}

Finally, we propose one more generalization of \thref{thm:08}, which extends the orthogonal variant of \thref{thm:14} to any $1\leq k<n-1$:

\begin{theorem}[cf. \thref{intro:05}]\thlabel{thm:22}
Let $K\subset\mathbb R^n$, $n\geq 3$, be a~convex body, and let $p\in\mathbb R^n$ be any point of the ambient space. If $p$ is a~$(-k)$-aligned quasi-center of $k$-reflection for $K$, $1\leq k<n-1$, then $K$ is a~body of $k$-revolution, with hyperaxis of $k$-revolution passing through $p$.
\end{theorem}

\noindent It is, in fact, an immediate consequence of an even more general theorem:

\begin{theorem}\thlabel{thm:21}
Let $K\subset\mathbb R^n$, $n\geq 3$, be a~convex body, and let $P\subset\mathbb R^n$ be any affine subspace. If there exists an affine subspace $T\supset P$ satisfying $k\equalscolon\dim(P+\vec T^\perp)<n-1$ and such that every intersection $K\cap H$ of $K$ with an affine hyperplane $H\supset P$ is invariant under the action of a~group conjugate to $\pi_k$ and fixing $P$, with hyperplane of reflection contained in $T\cap H$, then $K$ is a~body of $k$-revolution about the hyperaxis $P+\vec T^\perp$.
\end{theorem}

\noindent Note that the affine subspace $P$ plays here the role of a~point $p$ in the definition of alignment with respect to the affine subspace $T$.

\section{Passing from the affine to the orthogonal variant}\label{sec:06}

Although it is usually difficult to prove the logical equivalence between an affine and an orthogonal variant of \thref{con:04,con:05}, there are situations when it is actually possible. In this section, we will give such an example:

\begin{proposition}\thlabel{lem:01}
Let $K\subset\mathbb R^n$, $n\geq 4$, be a~convex body, and let $p\in\mathbb R^n$ be any point of the ambient space. If $p$ is a~pseudo-center of $1$-reflection for $K$ and $K$ itself is a~body of affine $1$-revolution, then $K$ is actually a~body of $1$-revolution or an ellipsoid.
\end{proposition}

\begin{proof}
Since $K$ is a~body of affine $1$-revolution, there exists a~linear map $\boldsymbol A\in\GL(\mathbb R^n)$ such that $\boldsymbol AK$ is a~body of $1$-revolution. Again, without loss of generality, we may assume that $\boldsymbol A$ is diagonal. Indeed, it admits a~singular value decomposition $\boldsymbol A=\boldsymbol U\boldsymbol D\boldsymbol V^\top$, where $\boldsymbol U,\boldsymbol V\in\O(\mathbb R^n)$ and $\boldsymbol D$ is a~diagonal matrix with non-negative entries, and we can consider $\boldsymbol U^\top\boldsymbol A$ instead of $\boldsymbol A$ and $\boldsymbol V^\top K$ instead of $K$.\\

Denote by $L\in\Gr_1(\mathbb R^n)$ the axis of affine $1$-revolution for $K$ and let $H+p$ be any affine hyperplane passing through $p$. Clearly, $P_2\colonequals\boldsymbol A^{-\top}L^\perp\cap\boldsymbol AH$ is the hyperplane of $1$-revolution and $\Lambda_2\colonequals\langle\boldsymbol AL,\boldsymbol A^{-\top}H^\perp\rangle\cap\boldsymbol AH$ is the direction of the axis of $1$-revolution for $\boldsymbol AK\cap(\boldsymbol AH+\boldsymbol Ap)$ (unless $\boldsymbol AL\perp\boldsymbol AH$, when the intersection is an ellipsoid). Hence $P_1\colonequals(\boldsymbol A^\top\boldsymbol A)^{-1}L^\perp\cap H$ is the hyperplane of affine $1$-revolution and $\Lambda_1\colonequals\langle L,(\boldsymbol A^\top\boldsymbol A)^{-1}H^\perp\rangle\cap H$ is the direction of the axis of affine $1$-revolution for $K\cap(H+p)$ (again, unless $\boldsymbol AL\perp\boldsymbol AH$, when the intersection is an ellipsoid).\\

Denote by $G_1<\GA(H+p)$ the corresponding symmetry group, affinely conjugate to $\rho_1$. By assumption, $K\cap(H+p)$ is invariant under the action of a~group $G_2<\GA(H+p)$, orthogonally conjugate to $\pi_1$. If $G_1\ngtr G_2$, then by \cite[Lemma~4]{10.2307/1968975}, $K\cap(H+p)$ is an ellipsoid. Suppose then that the intersection is not an ellipsoid. In particular, that $G_1>G_2$. In this case, $\Lambda_1$ must be perpendicular to $P_1$, i.e., equal to $P_1^\perp\cap H$.\\

Denote $L\equalscolon\langle\boldsymbol l\rangle$ and $H\equalscolon\langle\boldsymbol h\rangle^\perp$, where $\boldsymbol l,\boldsymbol h\in\mathbb S^{n-1}$ are unit vectors. The latter condition reads
\begin{equation}\label{eq:15}\langle\boldsymbol l,(\boldsymbol A^\top\boldsymbol A)^{-1}\boldsymbol h\rangle\cap H=\langle\boldsymbol A^\top\boldsymbol A\boldsymbol l,\boldsymbol h\rangle\cap H.\end{equation}
It is easy to see that the left-hand side of \eqref{eq:15} is spanned by
$$((\boldsymbol A^\top\boldsymbol A)^{-1}\boldsymbol h\cdot\boldsymbol h)\boldsymbol l-(\boldsymbol l\cdot\boldsymbol h)(\boldsymbol A^\top\boldsymbol A)^{-1}\boldsymbol h,$$
whereas the right-hand side of \eqref{eq:15} is spanned by
$$\boldsymbol A^\top\boldsymbol A\boldsymbol l-(\boldsymbol A^\top\boldsymbol A\boldsymbol l\cdot\boldsymbol h)\boldsymbol h.$$
Now, the condition
\begin{equation}\label{eq:16}((\boldsymbol A^\top\boldsymbol A)^{-1}\boldsymbol h\cdot\boldsymbol h)\boldsymbol l-(\boldsymbol l\cdot\boldsymbol h)(\boldsymbol A^\top\boldsymbol A)^{-1}\boldsymbol h\parallel\boldsymbol A^\top\boldsymbol A\boldsymbol l-(\boldsymbol A^\top\boldsymbol A\boldsymbol l\cdot\boldsymbol h)\boldsymbol h\end{equation}
can be expressed in terms of polynomial (more precisely, determinantal) equations. Hence, the subset of those $\boldsymbol h\in\mathbb S^{n-1}$ for which \eqref{eq:16} is satisfied is either the whole hypersphere or a~nowhere dense subset of the hypersphere.\\

In the first case, let $\boldsymbol h$ be an eigenvector of $\boldsymbol A^\top\boldsymbol A$ with eigenvalue $\lambda\in\mathbb R$. Then \eqref{eq:16} simplifies to
$$\boldsymbol l-(\boldsymbol l\cdot\boldsymbol h)\boldsymbol h\parallel\boldsymbol A^\top\boldsymbol A\boldsymbol l-(\boldsymbol l\cdot\boldsymbol h)\lambda\boldsymbol h=\boldsymbol A^\top\boldsymbol A(\boldsymbol l-(\boldsymbol l\cdot\boldsymbol h)\boldsymbol h).$$
It means that either $\boldsymbol l=\boldsymbol h$ is itself an eigenvector of $\boldsymbol A^\top\boldsymbol A$ or $\boldsymbol l-(\boldsymbol l\cdot\boldsymbol h)\boldsymbol h\neq\mathbf 0$ is an eigenvector of $\boldsymbol A^\top\boldsymbol A$ for every choice of $\boldsymbol h$.\\

In the first subcase, observe that $\langle\boldsymbol l\rangle^\perp$ is an invariant subspace of $\boldsymbol A$. For a~normalized vector $\boldsymbol h\colonequals\frac{\boldsymbol l+\boldsymbol u}{1+\|\boldsymbol u\|^2}$, where $\boldsymbol u\in\langle\boldsymbol l\rangle^\perp\setminus\{\mathbf 0\}$ is an arbitrary non-zero vector, the condition \eqref{eq:16} yields
$$((\boldsymbol A^\top\boldsymbol A)^{-1}\boldsymbol u\cdot\boldsymbol u)\boldsymbol l-(\boldsymbol A^\top\boldsymbol A)^{-1}\boldsymbol u\parallel\|\boldsymbol u\|^2\boldsymbol l-\boldsymbol u,$$
or, equivalently,
$$\boldsymbol l-\|\boldsymbol A^{-\top}\boldsymbol u\|^{-2}(\boldsymbol A^\top\boldsymbol A)^{-1}\boldsymbol u\parallel\boldsymbol l-\|\boldsymbol u\|^{-2}\boldsymbol u.$$
Hence
$$\|\boldsymbol A^{-\top}\boldsymbol u\|^{-2}(\boldsymbol A^\top\boldsymbol A)^{-1}\boldsymbol u=\|\boldsymbol u\|^{-2}\boldsymbol u$$
for every non-zero vector $\boldsymbol u\in\langle\boldsymbol l\rangle^\perp\setminus\{\mathbf 0\}$, which means that $\langle\boldsymbol l\rangle^\perp$ is actually an eigenspace of $\boldsymbol A$. But then it follows that $K=\boldsymbol A^{-1}(\boldsymbol AK)$ is indeed a~body of $1$-revolution, which concludes the proof.\\

In the second subcase, observe that there is an orthonormal basis $(\boldsymbol e_1,\boldsymbol e_2,\ldots,\boldsymbol e_n)$ of $\mathbb R^n$, consisting of eigenvectors of $\boldsymbol A^\top\boldsymbol A$ with eigenvalues $\lambda_1,\lambda_2,\ldots,\lambda_n$, respectively. By assumption, we know that $\boldsymbol u\colonequals\boldsymbol l-(\boldsymbol l\cdot\boldsymbol e_1)\boldsymbol e_1$ is an eigenvector of $\boldsymbol A^\top\boldsymbol A$ with eigenvalue $\lambda$. Since $\boldsymbol u$ is orthogonal to $\boldsymbol e_1$, it may be written as a~linear combination $a_2\boldsymbol e_2+a_3\boldsymbol e_3+\cdots+a_n\boldsymbol e_n$. Moreover, $a_i\neq 0$, since otherwise $\boldsymbol l-(\boldsymbol l\cdot\boldsymbol e_i)\boldsymbol e_i=\boldsymbol l$ would be an eigenvector of $\boldsymbol A^\top\boldsymbol A$, a~contradiction. Hence
$$a_2\lambda_2\boldsymbol e_2+a_3\lambda_3\boldsymbol e_3+\cdots+a_n\lambda_n\boldsymbol e_n=\boldsymbol A^\top\boldsymbol A\boldsymbol u=\lambda\boldsymbol u=a_2\lambda\boldsymbol e_2+a_3\lambda\boldsymbol e_3+\cdots+a_n\lambda\boldsymbol e_n,$$
which implies $\lambda_2=\lambda_3=\cdots=\lambda_n=\lambda$. Applying the same argument to $\boldsymbol e_2$ instead of $\boldsymbol e_1$ yields that $\boldsymbol A^\top\boldsymbol A=\lambda\mathbf I_n$. In particular, $\boldsymbol l$ is an eigenvector of $\boldsymbol A^\top\boldsymbol A$, a~contradiction.\\

Finally, in the second case, observe that $p$ is a~pseudo-center of affine $0$-revolution for $K$. But then, by \thref{thm:17}, $K$ is an ellipsoid, which concludes the proof.
\end{proof}

\begin{remark}
It is natural to ask whether \thref{lem:01} also holds for $n=3$. The only element of the proof that relies on the assumption $n\geq 4$ is \cite[Lemma~4]{10.2307/1968975}, which happens to fail for $n=3$. Indeed, a~$2$-dimensional convex body can have many axes of affine reflection without being an ellipse. Although we do believe that \thref{lem:01} can be extended, in our attempts to replace \cite[Lemma~4]{10.2307/1968975}, however, we are quickly getting on shaky ground. The case when all the intersections $K\cap H$ of $K$ with affine planes $H\in\Gr_2(\mathbb R^3)+p$ passing through $p$ admit more than one axis of affine reflection does not yield that $K$ is an ellipsoid anymore (the assumption is satisfied by, for example, a~symmetric lens centered at $p$). Hence \thref{lem:01} for $n=3$ becomes more geometric in nature than its purely algebraic variant for $n\geq 4$.
\end{remark}

If $K\subset\mathbb R^n$ is an ellipsoid, then, by the Principal Axis Theorem, every point $p\in\mathbb R^n$ of the ambient space is a~pseudo-center of $k$-reflection for $K$, $0\leq k<n-1$. Interestingly, the orthogonal variant of \thref{con:05} is non-trivial even if we \emph{a priori} assume that $K$ is an ellipsoid. The following proposition also explains why there is no dichotomy in the conclusion of \thref{con:05}, unlike the previous \thref{con:04}:

\begin{proposition}\thlabel{lem:07}
Let $K\subset\mathbb R^n$, $n\geq 3$, be an ellipsoid, and let $p\in\mathbb R^n$ be any point of the ambient space. If $p$ is a~pseudo-center of $k$-revolution for $K$, $0\leq k<n-2$, then $K$ itself is a~body of $k$-revolution.
\end{proposition}

\begin{proof}
We will even strengthen the theorem by assuming that every intersection $K\cap H$ of $K$ with an affine hyperplane $H$ in some open subset of $\Gr_{n-1}(\mathbb R^n)+p$ is invariant under the action of a~group conjugate to $\rho_k$. Without loss of generality, we may assume that $K$ is defined by the inequality $\boldsymbol x^\top\boldsymbol A\boldsymbol x\leq 1$, where $\boldsymbol A\in\GL(\mathbb R^n)$ is diagonal. Also, since the intersections of $K$ with any two parallel hyperplanes are similar, without loss of generality, we may assume that $p$ is the origin.\\

The intersection of $K$ with the hyperplane $\boldsymbol U\langle\boldsymbol e_n\rangle^\perp$, where $\boldsymbol U\in\O(\mathbb R^n)$, is an ellipsoid congruent to the intersection of $\boldsymbol U^{-1}K$ with the hyperplane $\langle\boldsymbol e_n\rangle^\perp$, and as such is defined by $\boldsymbol B\colonequals\boldsymbol P^\top\boldsymbol U^\top\boldsymbol A\boldsymbol U\boldsymbol P$, where $\boldsymbol P$ is the matrix of the natural embedding of $\mathbb R^{n-1}$ into $\mathbb R^n$. Hence, by assumption, the matrix $\boldsymbol B$ has an eigenvalue of multiplicity at least $n-1-k$.\\

Let us now compute the characteristic polynomial of $\boldsymbol B$. By Weinstein-Aronszajn identity (see \cite[(2)]{terrytao}), we obtain
\begin{align*}
&\det(\boldsymbol B-\lambda\mathbf I_{n-1})=\det(\boldsymbol P^\top\boldsymbol U^\top\boldsymbol A\boldsymbol U\boldsymbol P-\lambda\mathbf I_{n-1})=-\lambda^{-1}\det(\boldsymbol A\boldsymbol U\boldsymbol P\boldsymbol P^\top\boldsymbol U^\top-\lambda\mathbf I_n)\\
&\quad=-\lambda^{-1}\det(\boldsymbol A(\mathbf I_n-\boldsymbol u\boldsymbol u^\top)-\lambda\mathbf I_n)=-\lambda^{-1}\det(\boldsymbol A-\lambda\mathbf I_n-\boldsymbol A\boldsymbol u\boldsymbol u^\top)\\
&\quad=-\lambda^{-1}\det(\boldsymbol A-\lambda\mathbf I_n)\det(\mathbf I_n-(\boldsymbol A-\lambda\mathbf I_n)^{-1}\boldsymbol A\boldsymbol u\boldsymbol u^\top)=-\lambda^{-1}\det(\boldsymbol A-\lambda\mathbf I_n)(1-\boldsymbol u^\top(\boldsymbol A-\lambda\mathbf I_n)^{-1}\boldsymbol A\boldsymbol u),
\end{align*}
where $\boldsymbol u\colonequals\boldsymbol U\boldsymbol e_n$. Using the fact that $\boldsymbol A=\Diag(a_1,a_2,\ldots,a_n)$ is diagonal, we obtain that the latter is equal to
\begin{align*}
&-\lambda^{-1}\left(\prod_{i=1}^n(a_i-\lambda)\right)\left(1-\sum_{i=1}^nu_i^2\frac{a_i}{a_i-\lambda}\right)=-\lambda^{-1}\left(\prod_{i=1}^n(a_i-\lambda)\right)\sum_{i=1}^nu_i^2\left(1-\frac{a_i}{a_i-\lambda}\right)\\
&\quad=\left(\prod_{i=1}^n(a_i-\lambda)\right)\sum_{i=1}^nu_i^2\frac{1}{a_i-\lambda}=\sum_{i=1}^nu_i^2\left(\prod_{j\neq i}^n(a_j-\lambda)\right).
\end{align*}
Thus every convex combination of polynomials $p_1,p_2,\ldots,p_n$ given by
$$p_i(\lambda)\colonequals\prod_{j\neq i}^n(a_j-\lambda),$$
with coefficients in some open subset of the standard simplex $\triangle^n$, has a~root of multiplicity at least $n-k-1$.\\

To conclude the proof, we will need the following lemma:

\begin{lemma}
Let $f_1,f_2,\ldots,f_n\in\mathbb C[\lambda]$, $n\geq 2$, be polynomials, each convex combination of which, with coefficients in some open subset of the standard simplex $\triangle^n$, has a~root of multiplicity $m\geq 2$. Then all the polynomials have a~common root of multiplicity $m$.
\end{lemma}

\begin{proof}
Denote by $U\subseteq\triangle^n$ the open set of coefficients for which the assumption is met. The proof will be by induction over $n\in\mathbb N$. When $n=1$, there is nothing to prove. When $n=2$, for every $(\alpha,1-\alpha)\in U$ there exists $\lambda(\alpha)\in\mathbb C$ such that
$$\begin{pmatrix}f_1(\lambda(\alpha))&f_1'(\lambda(\alpha))&\ldots&f_1^{(m-1)}(\lambda(\alpha))\\f_2(\lambda(\alpha))&f_2'(\lambda(\alpha))&\ldots&f_2^{(m-1)}(\lambda(\alpha))\end{pmatrix}^\top\begin{pmatrix}\alpha\\1-\alpha\end{pmatrix}=\mathbf 0.$$
Now, if $\lambda(\alpha_1)=\lambda_*=\lambda(\alpha_2)$ for some $\alpha_1\neq\alpha_2$, then
$$\begin{pmatrix}f_1(\lambda_*)&f_1'(\lambda_*)&\ldots&f_1^{(m-1)}(\lambda_*)\\f_2(\lambda_*)&f_2'(\lambda_*)&\ldots&f_2^{(m-1)}(\lambda_*)\end{pmatrix}^\top=\mathbf 0$$
and the assertion follows. Otherwise, i.a., the first maximal minor
$$\begin{vmatrix}f_1(\lambda)&f_2(\lambda)\\f_1'(\lambda)&f_2'(\lambda)\end{vmatrix}$$
vanishes for infinitely many pairwise different values of $\lambda\in\mathbb C$, and since the determinant itself is a~polynomial in $\lambda$, the equality holds for every $\lambda\in\mathbb C$. In the complex domain $\{\lambda\in\mathbb C\mid f_2\neq 0\}$ it is equivalent to $[f_1(\lambda)/f_2(\lambda)]'\equiv 0$, which implies $f_1(\lambda)/f_2(\lambda)\equiv\Const$ Again, the assertion follows.\\

Finally, when $n\geq 3$, let $(\alpha_1,\alpha_2,\ldots,\alpha_n)\in U$ be any list of coefficients. Fix $\beta_*\in[0,1]$ such that $\alpha_1=(\alpha_1+\alpha_2)\beta_*$ and $\alpha_2=(\alpha_1+\alpha_2)(1-\beta_*)$, and consider the polynomials $\beta f_1+(1-\beta)f_2,f_3,\ldots,f_n$ for $\beta$ in some open neighborhood of $\beta_*$. By induction hypothesis, they all have a~common root $z(\beta)$ of order $m$. Since the set of common roots of $f_3,f_4,\ldots,f_n$ is finite, $z(\beta_1)=z_*=z(\beta_2)$ for some $\beta_1\neq\beta_2$. It follows that $z_*$ is also a~root of order $m$ of both $f_1$ and $f_2$, which concludes the proof.
\end{proof}

The above lemma immediately implies that the polynomials $p_1,p_2,\ldots,p_n$ have a~common root of multiplicity $n-k-1$. Hence, the characteristic polynomial of $\boldsymbol A$ has a~root of multiplicity $n-k$, which concludes the proof.
\end{proof}

An immediate corollary is a~refinement of \thref{thm:13}:

\begin{theorem}\thlabel{thm:15}
Let $K\subset\mathbb R^n$, $n\geq 4$, be a~convex body with boundary of class $C^3$ and let $p\in\mathbb R^n$ be any point of the ambient space. If $p$ is a~quasi-center of (affine) $1$-revolution for $K$, then $K$ itself is a~body of (affine) $1$-revolution.
\end{theorem}

\section{Proofs of the new theorems}\label{sec:04}

In this section, we will prove all the theorems formulated in \sref{sec:09}. But before, in the first subsection, we will consider problems of a~completely new kind, concerning many convex bodies simultaneously. Its main results --- \thref{lem:03,lem:05} --- will be heavily used in the following subsections, each containing the proof of a~single theorem, sometimes accompanied by its counterpart for a~different group of symmetries.

\subsection{Many-body problems}\label{sec:05}

Implementing the general proof paradigm described in \sref{sec:07}, we will need the following proposition to later prove an affine variant of \thref{thm:08}:

\begin{proposition}\thlabel{lem:03}
Let $K_1,K_2\subset\mathbb R^2$ be different convex bodies. If for all affine lines $L\in\Graff_1(\mathbb R^2)$ that intersect both $K_1$ and $K_2$, the midpoints of $K_1\cap L$ and $K_2\cap L$ coincide, then $K_1,K_2$ are concentric, homothetic ellipses.
\end{proposition}

\begin{proof}
First, observe that one of the bodies $K_1,K_2$ must contain the other in its interior. Indeed, every chord of $K_1$ with one endpoint in the intersection $\partial K_1\cap\partial K_2$ must also be a~chord of $K_2$, which implies $K_1=K_2$, unless the boundaries of $K_1$ and $K_2$ are disjoint. But since the interiors of $K_1$ and $K_2$ must have a~non-empty intersection, the assertion follows. Therefore, we may henceforth assume that $K_1\subset\Int K_2$.\\

\begin{claim}\thlabel{lem:06}
The boundaries of $K_1,K_2$ are twice differentiable everywhere.
\end{claim}

\begin{proof}
Let $p\in\partial K_2$ be any point on the boundary and let $\rho_1,\rho_2:\Gr_1(\mathbb R^2)+p\to\mathbb R$ be functions that assign to each line $L$ passing through $p$ the distance from $p$ to the nearest and farthest point of $K_1\cap L$, respectively. Since $K_1$ is convex, by Alexandrov's theorem \cite[Teorema]{Alex}, both functions $\rho_1,\rho_2$ are twice differentiable almost everywhere. In particular, there exists a~line $L_*$ passing through $p$ such that both $\rho_1,\rho_2$ are twice differentiable at $L_*$. Now, let $\tilde p\in\partial K_2$ be the point of the intersection $\partial K_2\cap L_*$ other than $p$ and let $\tilde\rho_1,\tilde\rho_2:\Gr_1(\mathbb R^2)+\tilde p\to\mathbb R$ be analogous functions defined for the point $\tilde p$. Clearly, both $\tilde\rho_1,\tilde\rho_2$ are again twice differentiable at $L_*$. Furthermore, since the midpoints of $K_1\cap L$ and $K_2\cap L$ coincide, we have $\rho_1(L_*)=\tilde\rho_1(L_*)$, and therefore $\tilde\rho_1+\tilde\rho_2$ is the local parametrization of $\partial K_2$ at $p$. In particular, $\partial K_2$ is twice differentiable at $p$. Because the choice of $p\in\partial K_2$ was arbitrary, it follows that $\partial K_2$ is indeed twice differentiable everywhere. We show the same for $\partial K_1$ in a~completely analogous way.
\end{proof}

Let $I$ be the projection of $K_1$ onto the horizontal axis and denote by $\mu:I\to\mathbb R^2$ a~regular parametrization $\mu(x)\equalscolon(x,m(x))$ of the common midcurve of $K_1$ and $K_2$ in the direction of the vertical axis. Also, let $\gamma_{\pm i}:I\to\mathbb R^2$, $i=1,2$, be regular curves $\gamma_{\pm i}(x)\equalscolon(x,m(x)\pm g_i(x))$ parametrizing the upper and the lower arc of the boundary of $K_1$ and $K_2$, respectively. By \thref{lem:06}, we know that all five curves $\gamma_{+2},\gamma_{+1},\mu,\gamma_{-1},\gamma_{-2}$ are twice differentiable everywhere.\\

Now, fix any point in the interval $I$. After applying a~suitable shift, without loss of generality, we may assume that it is the origin. Denote by $u_\pm:I^2\to\mathbb R$ the first coordinate of the point of intersection of the line passing through $\gamma_{+1}(s)$ and $\gamma_{-1}(t)$ with the upper and the lower arc of the boundary of $K_2$, respectively. On some neighbourhood of $(0,0)$, the function is well-defined and satisfies an implicit equation
$$\det(\gamma_{\pm 2}(u_\pm(s,t))-\gamma_{-1}(t),\gamma_{+1}(s)-\gamma_{-1}(t))=0.$$
Hence, by the Implicit Function Theorem, there exists an open neighborhood of $(0,0)$ such that the unique solution $u_\pm$ of differentiability class $C^1$. Furthermore, its partial derivatives are given by
\begin{equation}\label{eq:09}
\begin{aligned}
u_\pm^{(1,0)}(s,t)&=\frac{g_1(t)\pm g_2(u_\pm)-m(t)+m(u_\pm)+t g_1'(s)-u_\pm g_1'(s)+t m'(s)-u_\pm m'(s)}{g_1(s)+g_1(t)+m(s)-m(t)\mp s g_2'(u_\pm)\pm t g_2'(u_\pm)-s m'(u_\pm)+t m'(u_\pm)},\\
u_\pm^{(0,1)}(s,t)&=\frac{g_1(s)\mp g_2(u_\pm)+m(s)-m(u_\pm)+s g_1'(t)-u_\pm g_1'(t)-s m'(t)+u_\pm m'(t)}{g_1(s)+g_1(t)+m(s)-m(t)\mp s g_2'(u_\pm)\pm t g_2'(u_\pm)-s m'(u_\pm)+t m'(u_\pm)}.
\end{aligned}
\end{equation}
The hypothesis is equivalent to
\begin{equation}\label{eq:10}
\gamma_{+1}(s)+\gamma_{-1}(t)-\gamma_{+2}(u_+(s,t))-\gamma_{-2}(u_-(s,t))=\mathbf 0.
\end{equation}
Differentiating \eqref{eq:10} with respect to $s$, substituting \eqref{eq:09}, then differentiating the result with respect to $t$, again substituting \eqref{eq:09} and finally evaluating the result at $(s,t)=(0,0)$ yields
$$\begin{pmatrix}-\frac{g_1(0) g_1'(0)-g_2(0) g_2'(0)}{g_1(0)^2},&-\frac{2 g_1(0) g_1'(0) m'(0)-2 g_2(0) g_2'(0) m'(0)+g_1(0)^2 m''(0)-g_2(0)^2 m''(0)}{2 g_1(0)^2}\end{pmatrix}=\mathbf 0,$$
whence we immediately obtain $m''(0)=0$. But since we may arbitrarily shift the coordinate system, $m''$ vanishes identically. It follows that every midcurve of $K_1$ is a~straight line, which by a~classical result of W.~Blaschke \cite[\S 9]{blaschke1923vorlesungen} implies that $K_1$ is an ellipse.\\

Since the hypothesis is affine-invariant, without loss of generality, we may assume that $K_1$ is a~disk centered at $o$. Let $p\in\partial K_2$ be any point on the boundary and let $m$ be the midpoint of any chord $p\tilde p$ of $K_2$ that intersects $K_1$. Since $m$ is also a~midpoint of some chord of $K_1$, the angle $\angle pmo$ is right, whence the triangle $\triangle po\tilde p$ is isosceles. Consequently, the geometric locus of all such points $\tilde p$ is a~circle centered at $o$ and passing through $p$. Because $p$ was arbitrary, $K_2$ must likewise be a~disk centered at $o$. This concludes the proof.
\end{proof}

The following conjecture, recently proven by F.~Nazarov, D.~Ryabogin, V.~Yaskin, and B.~Zawalski for centrally symmetric convex bodies with boundary of differentiability class $C^1$ \cite{NRYZ}, is a~higher-dimensional generalization of \thref{lem:03}:

\begin{conjecture}\thlabel{con:07}
Let $K_1,K_2\subset\mathbb R^n$, $n\geq 2$, be different convex bodies. If for all affine hyperplanes $H\in\Graff_{n-1}(\mathbb R^n)$ that intersect both $K_1$ and $K_2$, the centroids of $K_1\cap H$ and $K_2\cap H$ coincide, then $K_1,K_2$ are concentric, homothetic ellipsoids.
\end{conjecture}

For the sake of completeness, we will independently confirm \thref{con:07} here when both $K_1$ and $K_2$ are \emph{a priori} known to be ellipsoids. In this simple case, however, we require a~much smaller family of hyperplanes $H$ for which the centroids of $K_1\cap H$ and $K_2\cap H$ coincide:

\begin{lemma}\thlabel{lem:05}
Let $K_1,K_2\subset\mathbb R^n$, $n\geq 2$, be ellipsoids, let $\mathcal F\subseteq\Graff_{n-1}(\mathbb R^n)$ be a~family of affine hyperplanes that intersect both $K_1$ and $K_2$, containing a~hyperplane at every direction, and let $\mathcal N_0\subseteq\mathbb P(\mathbb R^n)$ be the family of lines normal to the affine hyperplanes $H\in\mathcal F$ passing through the centroid of $K_1$. If for all affine hyperplanes $H\in\mathcal F$ the centroids of $K_1\cap H$ and $K_2\cap H$ coincide, then the ellipsoids $K_1,K_2$ are concentric. Furthermore, the ellipsoids $K_1,K_2$ are homothetic, unless $\mathbb P(\mathbb R^n)\setminus\mathcal N_0$ is contained in a~finite union of proper linear subspaces.
\end{lemma}

\begin{proof}
Since the hypothesis is affine-invariant, without loss of generality we may assume that $K_1=\mathbb B^n(\mathbf 0,1)$ is a~Euclidean unit ball and $K_2=\boldsymbol AK_1+\boldsymbol b$, $\boldsymbol A\in\GL(\mathbb R^n),\ \boldsymbol b\in\mathbb R^n$. Denote by $\boldsymbol\mu(X)$ the centroid of $X$. Clearly,
$$\boldsymbol\mu(K_1\cap(\langle\boldsymbol\xi\rangle^\perp+t\boldsymbol\xi))=t\boldsymbol\xi.$$
It is also not difficult to see that
\begin{align*}
&\boldsymbol\mu(K_2\cap(\langle\boldsymbol\xi\rangle^\perp+t\boldsymbol\xi))=\boldsymbol\mu((\boldsymbol AK_1+\boldsymbol b)\cap(\langle\boldsymbol\xi\rangle^\perp+t\boldsymbol\xi))=\boldsymbol\mu(\boldsymbol AK_1\cap(\langle\boldsymbol\xi\rangle^\perp+(t\boldsymbol\xi-\boldsymbol b)))+\boldsymbol b\\
&\quad=\boldsymbol A\boldsymbol\mu(K_1\cap\boldsymbol A^{-1}(\langle\boldsymbol\xi\rangle^\perp+(t\boldsymbol\xi-\boldsymbol b)))+\boldsymbol b=\boldsymbol A\boldsymbol\mu(K_1\cap(\langle\boldsymbol A^\top\boldsymbol\xi\rangle^\perp+\boldsymbol A^{-1}(t\boldsymbol\xi-\boldsymbol b)))+\boldsymbol b\\
&\quad=\boldsymbol A((\boldsymbol A^{-1}(t\boldsymbol\xi-\boldsymbol b)\cdot\|\boldsymbol A^\top\boldsymbol\xi\|^{-1}\boldsymbol A^\top\boldsymbol\xi)\|\boldsymbol A^\top\boldsymbol\xi\|^{-1}\boldsymbol A^\top\boldsymbol\xi)+\boldsymbol b=\frac{t-\boldsymbol\xi^\top\boldsymbol b}{\boldsymbol\xi^\top\boldsymbol A\boldsymbol A^\top\boldsymbol\xi}\boldsymbol A\boldsymbol A^\top\boldsymbol\xi+\boldsymbol b.
\end{align*}
Now, by assumption, we have
\begin{align}
\nonumber&\forall\boldsymbol\xi\in\mathbb S^{n-1}\ \exists t\in\mathbb R\ \frac{t-\boldsymbol\xi^\top\boldsymbol b}{\boldsymbol\xi^\top\boldsymbol A\boldsymbol A^\top\boldsymbol\xi}\boldsymbol A\boldsymbol A^\top\boldsymbol\xi+\boldsymbol b=t\boldsymbol\xi\\
\nonumber\iff&\forall\boldsymbol\xi\in\mathbb S^{n-1}\ \exists t\in\mathbb R\ \frac{\boldsymbol\xi\cdot(t\boldsymbol\xi-\boldsymbol b)}{\boldsymbol\xi\cdot\boldsymbol A\boldsymbol A^\top\boldsymbol\xi}\boldsymbol A\boldsymbol A^\top\boldsymbol\xi=t\boldsymbol\xi-\boldsymbol b\\
\label{eq:18}\iff&\forall\boldsymbol\xi\in\mathbb S^{n-1}\ \exists s,t\in\mathbb R\ s\boldsymbol A\boldsymbol A^\top\boldsymbol\xi=t\boldsymbol\xi-\boldsymbol b\\
\label{eq:17}\iff&\forall\boldsymbol\xi\in\mathbb S^{n-1}\ \boldsymbol b\in\langle\boldsymbol\xi,\boldsymbol A\boldsymbol A^\top\boldsymbol\xi\rangle.
\end{align}
In particular, \eqref{eq:17} holds for every eigenvector $\boldsymbol\xi$ of $\boldsymbol A\boldsymbol A^\top$ and hence $\boldsymbol b$ is parallel to $\boldsymbol\xi$. But since $\boldsymbol A\boldsymbol A^\top$ admits $n\geq 2$ linearly independent eigenvectors, it follows that $\boldsymbol b=\mathbf 0$. This concludes the first part of the proof. For the second part, observe that \eqref{eq:18} reads
$$\forall\boldsymbol\xi\notin\mathcal N_0\ \exists s\in\mathbb R\ s\boldsymbol A\boldsymbol A^\top\boldsymbol\xi=\boldsymbol\xi,$$
which means that every $\boldsymbol\xi\notin\mathcal N_0$ is an eigenvector of $\boldsymbol A\boldsymbol A^\top$. It follows that $\mathbb P(\mathbb R^n)\setminus\mathcal N_0$ is contained in the union of eigenspaces of $\boldsymbol A\boldsymbol A^\top$. If $\boldsymbol A\boldsymbol A^\top=\lambda^2\mathbf I_n$ for some $\lambda\in\mathbb R$, then $\boldsymbol A=\lambda\boldsymbol U$, where $\boldsymbol U\in\O(\mathbb R^n)$, and hence $K_2=\lambda\mathbb B^n(\mathbf 0,1)$ is homothetic to $K_1$. Otherwise, the eigenspaces of $\boldsymbol A\boldsymbol A^\top$ are proper linear subspaces. This concludes the proof.
\end{proof}

Finally, in order to apply our paradigm to a~$(-1)$-aligned quasi-center of $k$-reflection (resp. $k$-revolution) for $k\geq 2$ or a~$(-1)$-aligned pseudo-center of $k$-reflection (resp. $k$-revolution) for $k\geq 0$, we would like to see the following question resolved:

\begin{question}\thlabel{con:09}
Let $K_1,K_2,\ldots,K_m\subset\mathbb R^n$, $m\geq 2$, $n\geq 2$, be pairwise different convex bodies. If for all affine hyperplanes $H\in\Graff_{n-1}(\mathbb R^n)$ that intersect all $K_i$, the centroids of $K_i\cap H$ are contained in some proper affine subspace of $H$ of dimension $0\leq k<m-1$, are all $K_i$ bodies of affine $k$-revolution, invariant under the action of the same group affinely conjugate to $\rho_k$?
\end{question}

\noindent Its weaker form, when $K_1,K_2,\ldots,K_m$ are \emph{a priori} known to be bodies of (affine) $k$-revolution, would also be interesting and useful:

\begin{question}\thlabel{con:10}
Let $K_1,K_2,\ldots,K_m\subset\mathbb R^n$, $m\geq 2$, $n\geq 2$, be pairwise different convex bodies of (affine) $k$-revolution, $0\leq k<n-1$. If for all affine hyperplanes $H\in\Graff_{n-1}(\mathbb R^n)$ that intersect all $K_i$, the intersections $K_i\cap H$ are invariant under the action of the same group (affinely conjugate to) $\rho_k$, are all $K_i$ bodies of affine $k$-revolution, invariant under the action of the same group (affinely conjugate to) $\rho_k$?
\end{question}

\noindent To the best of our knowledge, both are open problems. For $k=0$, \thref{con:09} reduces to \thref{con:07}. Note that \thref{con:09} encapsulates the whole difficulty of our approach. But since we are not yet able to prove it in its full generality, to apply either \thref{lem:05} or even \thref{con:10} (i.e., its rudimentary forms), we still need to prove independently that the intersections of $K$ with all the affine hyperplanes parallel to $T$ are bodies of (affine) $k$-revolution. Given the current state of knowledge, this unfortunately limits us to the case $k=1$.

\subsection{Proof of \thref{thm:14}}

First, we will prove the affine variant of \thref{thm:14}.

\begin{proof}[Proof of the affine variant of \thref{thm:14} for $n\geq 4$]
By the argument outlined in \sref{sec:07}, we have a~continuous family of convex bodies $(K_\tau)_{\tau\in\mathbb R}$, whose hyperplanar sections $K_\tau\cap H'$ for any fixed $H'\in\Graff_{n-2}(T)$ are all invariant under the action of the same group affinely conjugate to $\pi_0$ (i.e., are all centrally symmetric with respect to the same point). Hence, by, for example, \thref{thm:12}, every $K_\tau$ is an ellipsoid. Furthermore, since the common center of symmetry is, in particular, the common centroid of all the hyperplanar sections $K_\tau\cap H'$, by \thref{lem:05}, all the ellipsoids $K_\tau$ are concentric and homothetic with respect to their common centroid $p'\in T$. Now, we may apply an affine change of coordinate system so that the hyperplane $T$ is perpendicular to the line $\langle p,p'\rangle$ and the convex bodies $K_\tau$ are Euclidean balls centered at $p'$. It follows that all the intersections $K\cap T'$ of $K$ with affine hyperplanes $T'\in\Graff_{n-1}(\mathbb R^n)$ parallel to $T$ are Euclidean balls centered at $\langle p,p'\rangle\cap T'$, whence $K$ is a~body of $1$-revolution about the axis $\langle p,p'\rangle$. This concludes the proof.
\end{proof}

The above argument fails for $n=3$ because none of the theorems mentioned in \sref{sec:08} hold for $2$-dimensional convex bodies. Therefore, we have to make a~more profound use of the fact that all $K_\tau\cap H'$ are invariant under the action of the same group.

\begin{proof}[Proof of the affine variant of \thref{thm:14} for $n=3$]
This time, we have a~continuous family of planar convex bodies $(K_\tau)_{\tau\in\mathbb R}$, whose sections $K_\tau\cap L$ share the same midpoint for any fixed line $L\in\Graff_1(T)$. Hence by \thref{lem:03}, they are concentric, homothetic ellipses. We conclude the proof in the same way as before.
\end{proof}

The orthogonal variant of \thref{thm:14} for $n=3$ follows along the lines of \thref{thm:08}, whereas for $n\geq 4$ it is an immediate consequence of \thref{lem:01}. Actually, it would likewise follow from \thref{thm:08} by simple induction on the dimension $n$, if only \thref{thm:15} did not require an additional smoothness assumption.

\subsection{Proofs of \thref{thm:16,thm:02}}\label{sec:17}

First, we will prove \thref{thm:16}.

\begin{proof}[Proof of \thref{thm:16}]
Let $T'\in\Graff_3(\mathbb R^4)$ be any affine hyperplane parallel to $T$. By the argument outlined in \sref{sec:07}, we have a~continuous family of convex bodies $(K_\tau)_{\tau\in\mathbb R}$ contained in $T'$, whose hyperplanar sections $K_\tau\cap(\sigma(H)+p)$ for any fixed $H\in\Gr_2(T)$ are all invariant under the action of the same group affinely conjugate to $\pi_0$ (i.e., are all centrally symmetric with respect to the same point). Hence, by \thref{thm:05}, $K_\tau$ is itself centrally symmetric for every $\tau$ such that $(\sigma(H_1)+p)\cap(\sigma(H_2)+p)\cap\Int K_\tau\neq\emptyset$ for every $H_1,H_2\in\Gr_2(T)$.

\begin{claim}\thlabel{lem:08}
There exists a~blunt (i.e., not containing its vertex) right circular cone $C$ with axis $T^\perp$ such that $\sigma(H_1)\cap\sigma(H_2)\cap C\neq\emptyset$ for every $H_1,H_2\in\Gr_2(T)$.
\end{claim}

\begin{proof}
Without loss of generality, we may assume that $T=\langle\boldsymbol e_n\rangle^\perp$. Let $H_1,H_2\in\Gr_2(T)$ be any pair of distinct hyperplanes (see \fref{fig:06}). Let $\boldsymbol n_1,\boldsymbol n_2$ be the corresponding unit normal vectors. Denote by $\beta$ the angle between $\boldsymbol n_1$ and $\boldsymbol n_2$. Without loss of generality, we may assume that $\beta\in(0,\frac{\pi}{2}]$. Also, let $\alpha_1,\alpha_2$ be the angles between $\sigma(H_1),\sigma(H_2)$ and $T^\perp$, respectively. Finally, denote by $\tilde{\boldsymbol n}_1\colonequals(\cos\alpha_1\boldsymbol n_1,-\sin\alpha_1),\tilde{\boldsymbol n}_2\colonequals(\cos\alpha_2\boldsymbol n_2,-\sin\alpha_2)$ the unit normal vectors of $\sigma(H_1),\sigma(H_2)$, respectively.\\

\begin{figure}
\includegraphics{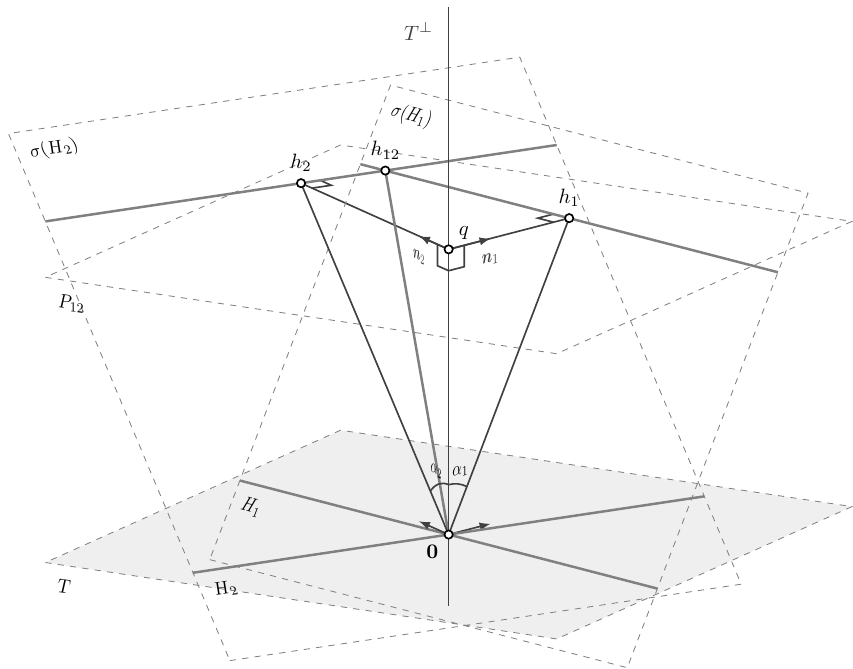}
\caption{Notations used in the proof of \thref{lem:08}}
\label{fig:06}
\end{figure}

The canonical metric $d$ on the Grassmannian of $1$-codimensional hyperplanes is simply the measure of the dihedral angle (cf. \cite[\S 4]{Wong}). Since the map $\sigma:\Gr_2(T)\to\Gr_3(\mathbb R^4)$ is Lipschitz continuous, we have that
$$\frac{d(\sigma(H_1),\sigma(H_2))}{d(H_1,H_2)}=\frac{\angle\tilde{\boldsymbol n}_1\tilde{\boldsymbol n}_2}{\angle\boldsymbol n_1\boldsymbol n_2}$$
is bounded. This, in turn, is equivalent to the fact that
$$\frac{1-\cos\angle\tilde{\boldsymbol n}_1\tilde{\boldsymbol n}_2}{1-\cos\angle\boldsymbol n_1\boldsymbol n_2}=\frac{1-\tilde{\boldsymbol n}_1\cdot\tilde{\boldsymbol n}_2}{1-\boldsymbol n_1\cdot\boldsymbol n_2}$$
is bounded. Indeed,
$$\frac{1-\cos\angle\tilde{\boldsymbol n}_1\tilde{\boldsymbol n}_2}{1-\cos\angle\boldsymbol n_1\boldsymbol n_2}=\frac{1-\cos\angle\tilde{\boldsymbol n}_1\tilde{\boldsymbol n}_2}{(\angle\tilde{\boldsymbol n}_1\tilde{\boldsymbol n}_2)^2}\cdot\frac{(\angle\boldsymbol n_1\boldsymbol n_2)^2}{1-\cos\angle\boldsymbol n_1\boldsymbol n_2}\cdot\frac{(\angle\tilde{\boldsymbol n}_1\tilde{\boldsymbol n}_2)^2}{(\angle\boldsymbol n_1\boldsymbol n_2)^2}$$
is a~product of bounded functions. Furthermore, we can rewrite
$$\frac{1-\tilde{\boldsymbol n}_1\cdot\tilde{\boldsymbol n}_2}{1-\boldsymbol n_1\cdot\boldsymbol n_2}=\frac{1-\cos\alpha_1\cos\alpha_2\cos\beta-\sin\alpha_1\sin\alpha_2}{1-\cos\beta}=1-\sin\alpha_1\sin\alpha_2+\frac{1-\cos(\alpha_1-\alpha_2)}{1-\cos\beta}\cos\beta,$$
which implies that
$$\frac{1-\cos(\alpha_1-\alpha_2)}{1-\cos\beta}$$
is bounded, and, in consequence, that
$$\frac{\sin^2(\alpha_1-\alpha_2)}{\sin^2\beta}=\frac{1-\cos(\alpha_1-\alpha_2)}{1-\cos\beta}\cdot\frac{1+\cos(\alpha_1-\alpha_2)}{1+\cos\beta}$$
is bounded.\\

Now, let $q\in T^\perp$ be any point in the orthogonal complement of $T$ and let $P_{12}=\langle\boldsymbol n_1,\boldsymbol n_2\rangle+q$ be the $2$-dimensional affine plane parallel to both $\boldsymbol n_1$ and $\boldsymbol n_2$ and passing through $q$ (see \fref{fig:06}). Denote by $h_1,h_2,h_{12}$ the intersections $\sigma(H_1)\cap(\langle\boldsymbol n_1\rangle+q),\sigma(H_2)\cap(\langle\boldsymbol n_2\rangle+q),\sigma(H_1)\cap\sigma(H_2)\cap(\langle\boldsymbol n_1,\boldsymbol n_2\rangle+q)$, respectively. Observe that
$$\tan d(\sigma(H_1)\cap\sigma(H_2),T^\perp)=\frac{\|h_{12}-q\|}{\|q\|}.$$
Since $\|h_{12}-q\|$ happens to be the diameter of the circle circumscribed on the quadrilateral $\square qh_1h_{12}h_2\subset P_{12}$, its length can be computed using the laws of sines and cosines:
$$\|h_{12}-q\|^2=\frac{\|h_1-h_2\|^2}{\sin^2\beta}=\frac{\|h_1-q\|^2+\|h_2-q\|^2-2\|h_1-q\|\|h_2-q\|\cos\beta}{\sin^2\beta}.$$
Hence
\begin{align*}
&\frac{\|h_{12}-q\|^2}{\|q\|^2}=\frac{\tan^2\alpha_1+\tan^2\alpha_2-2\tan\alpha_1\tan\alpha_2\cos\beta}{\sin^2\beta}=\frac{(\tan\alpha_1-\tan\alpha_2)^2+2\tan\alpha_1\tan\alpha_2(1-\cos\beta)}{\sin^2\beta}\\
&\quad=\frac{1}{\cos^2\alpha_1\cos^2\alpha_2}\cdot\frac{\sin^2(\alpha_1-\alpha_2)}{\sin^2\beta}+2\tan\alpha_1\tan\alpha_2\cdot\frac{1-\cos\beta}{\sin^2\beta}
\end{align*}
is bounded. Indeed, since we assumed that $\sigma(H)\cap T=H$, both $\alpha_1$ and $\alpha_2$ are different from $\frac{\pi}{2}$, and since $\sigma$ is a~continuous map defined on a~compact domain $\Gr_2(T)$, they are actually bounded away from $\frac{\pi}{2}$. Thus, all the functions
$$\frac{1}{\cos^2\alpha_1},\quad\frac{1}{\cos^2\alpha_2},\quad\frac{\sin^2(\alpha_1-\alpha_2)}{\sin^2\beta},\quad\tan\alpha_1,\quad\tan\alpha_2,\quad\frac{1-\cos\beta}{\sin^2\beta}$$
are bounded. It follows that $d(\sigma(H_1)\cap\sigma(H_2),T^\perp)$ is itself bounded away from $\frac{\pi}{2}$. This concludes the proof.
\end{proof}

Since $p$ is an interior point of $K$, there exists an open slab $S$ parallel to $T$ such that $C\cap S\subset K\cap S$. In particular, there exists $\tau_*\in\mathbb R$ such that the corresponding convex body $K_{\tau_*}\subset T'$ satisfies the assumption of \thref{thm:05}, whence it follows that it is centrally symmetric with respect to some point $c_*\in T'$. Now, consider any convex body $K_\tau$, $\tau\in\mathbb R$, and let $L\in\Gr_1(T')+c_*$ be any line passing through $c_*$. The continuous map $\sigma:\Gr_2(T)\to\Gr_3(\mathbb R^4)$ gives rise to an odd, continuous map $\delta:\mathbb S^2\to\mathbb R$, defined as in \thref{thm:04} with respect to the origin $c_*$. The restriction of $\delta$ to the equator $\mathbb S^2\cap L^\perp\cong\mathbb S^1$ is again an odd, continuous map, and hence it attains the value $0$. It follows that there exists a~hyperplane $H\in\Gr_2(T)$ such that $\sigma(H)+p$ contains $L$. Since $c_*$ is the center of symmetry for $K_{\tau_*}$, it must be also the center of symmetry for $K_{\tau_*}\cap(\sigma(H)+p)$. But the latter coincides with the center of symmetry for $K_\tau\cap(\sigma(H)+p)$, which implies that $c_*$ is the midpoint of $K_\tau\cap L$. Because $L$ was arbitrary, $c_*$ turns out to be the common center of symmetry for all the convex bodies $(K_\tau)_{\tau\in\mathbb R}$. This concludes the proof.
\end{proof}

\begin{remark}
The above argument is quite general. All the seemingly superfluous assumptions (i.a., the restriction on the dimension) were dictated by \thref{thm:05}. Although the latter can most likely be generalized to an arbitrary dimension $n\geq 3$, the strong topological condition on the family of hyperplanes remains crucial for the proof of both \thref{thm:05} and \thref{thm:16}. In particular, it entails that the family of hyperplanes $\sigma(\Gr_2(T))$ thoroughly sweeps the entire $K$.
\end{remark}

\begin{remark}
The same argument proves \thref{thm:16} for any point $p\in\mathbb R^4$ and any hyperplane $T\in\Gr_3(\mathbb R^4)$ such that $K-p$ absorbs the union of two open half-spaces $\mathbb R^4\setminus T$ (i.e., $\bigcup_{0<r<\infty}r(K-p)\supseteq\mathbb R^4\setminus T$). In particular, \thref{thm:16} holds for any smooth point $p\in\partial K$ and $T$ being the hyperplane tangent to $\partial K$ at $p$.
\end{remark}

Secondly, we will prove \thref{thm:02}.

\begin{proof}[Proof of \thref{thm:02}]
The argument is along the same lines. Since $\sigma$ is a~continuous map defined on a~compact domain $\Gr_{n-2}(T)$, the angle between $\sigma(H)$ and $T^\perp$ is bounded away from $\frac{\pi}{2}$ and hence there exists a~blunt right circular cone $C$ with axis $T^\perp$ such that $\sigma(H)\cap C\neq\emptyset$ for every $H\in\Gr_{n-2}(T)$. Furthermore, since $p$ is an interior point of $K$, there exists an open slab $S$ parallel to $T$ such that $C\cap S\subset K\cap S$. In particular, there exists $\tau_*\in\mathbb R$ such that the corresponding convex body $K_{\tau_*}\subset T'$ satisfies the assumption of \thref{thm:20}, whence it follows that it is an ellipsoid centered at some point $c_*\in T'$. The continuous map $\sigma:\Gr_{n-2}(T)\to\Gr_{n-1}(\mathbb R^n)$ gives rise to an odd, continuous map $\delta:\mathbb S^{n-2}\to\mathbb R$, defined as in \thref{thm:04} with respect to the origin $c_*$, which attains the value $0$. It follows that there exists a~hyperplane $H_*\in\Gr_{n-2}(T)$ such that $\sigma(H_*)+p$ contains $c_*$. Now, consider any convex body $K_\tau$, $\tau\in\mathbb R$, and let $L\in\Gr_1(T')+c_*$ be any line passing through $c_*$. By the initial symmetry assumption, the intersections $K_{\tau_*}\cap(\sigma(H_*)+p)$ and $K_\tau\cap(\sigma(H_*)+p)$ are ellipsoids homothetic with respect to the common center $c_*$, with some ratio $\lambda_*$. Again, there exists a~hyperplane $H\in\Gr_{n-2}(T)$ such that $\sigma(H)+p$ contains $L$. The intersections $K_{\tau_*}\cap(\sigma(H)+p)$ and $K_\tau\cap(\sigma(H)+p)$ are likewise ellipsoids homothetic with respect to the common center $c_*$, with some ratio $\lambda$. But on the subspace $(\sigma(H_*)+p)\cap(\sigma(H)+p)$ the ratio is already known to be $\lambda_*$, which implies that $\lambda=\lambda_*$ and the intersections $K_{\tau_*}\cap L$ and $K_\tau\cap L$ are homothetic with respect to the common center $c_*$, with ratio $\lambda_*$. Because $L$ was arbitrary, all the convex bodies $(K_\tau)_{\tau\in\mathbb R}$ turn out to be ellipsoids centered at $c_*$ and homothetic to $K_{\tau_*}$. This concludes the proof.
\end{proof}

\subsection{Proof of \thref{thm:21}}

Finally, we will prove \thref{thm:21}, from which \thref{thm:22} is a~straightforward corollary.

\begin{proof}[Proof of \thref{thm:21}]
The argument is along the lines of \thref{thm:08}. First, if $\dim T=n$, then, by assumption, every intersection $K\cap H$ of $K$ with an affine hyperplane $H\supset P$ is a~body of $k$-reflection through the hyperaxis $P$, whence $K$ itself is a~body of $k$-reflection through the hyperaxis $P$ and the assertion follows. Hence, without loss of generality, we may assume that $\dim(T)\leq n-1$.\\

Fix any affine hyperplane $H_*\supset P+\vec T^\perp$ and any affine line $L$ perpendicular to $H_*$ and intersecting it at some point $p_2\in H_*\setminus(P+\vec T^\perp)\setminus T$. Let $p_0\in P$ and $p_1\in P+\vec T^\perp$ be affinely independent with $p_2$. Let $\vec N$ be the normal space to the affine line spanned by $p_0,p_2$ in the affine plane spanned by $p_0,p_1,p_2$, and let $H\colonequals P+\vec N^\perp$. It is easy to see that $L\subset H$. Indeed, they are both perpendicular to $H_*$ and contain $p_2$. It follows that $L$ is parallel to $\vec P^\perp\cap\vec T\cap\vec H$.\\

Now, by assumption, $K\cap H$ is a~body of $k$-reflection, with hyperaxis of $k$-reflection containing $P$ and hyperplane of $k$-reflection contained in $T\cap H$. But since $P\subset T\cap H$ and the hyperaxis of $k$-reflection is orthogonal to the hyperplane of $k$-reflection, the tangent space to the latter is contained in $\vec P^\perp\cap\vec T\cap\vec H$. Furthermore, since $p_2\in H\setminus T$, we have $H\neq T$ and consequently
\begin{align*}
\dim(\vec P^\perp\cap\vec T\cap\vec H)&=\dim(\vec T\cap\vec H)-\dim(\vec P)=(\dim(\vec T)-1)-\dim(\vec P)=(n-\dim(\vec T^\perp))-1-\dim(\vec P)\\
&=(n-1)-\dim(\vec P+\vec T^\perp)=(n-1)-k.
\end{align*}
A simple comparison of dimensions shows that the tangent space to the hyperplane of $k$-reflection is actually equal to $\vec P^\perp\cap\vec T\cap\vec H$. Hence, the hyperaxis of $k$-reflection is equal to the orthogonal complement of the latter in $H$, namely $P+\langle\vec T^\perp,\vec N\rangle\cap\vec N^\perp$.\\

Since $p_1-p_0\notin\vec N$, we have $p_2-p_0\in\langle\vec T^\perp,\vec N\rangle\cap\vec N^\perp$ and thus $p_2$ belongs to the hyperaxis of $k$-reflection. Now, recall that $L$ is parallel to the hyperplane of $k$-reflection, whence $K\cap L$ is invariant under reflection through the hyperplane $H_*$. Because the choice of $L$ from a~dense open subset was arbitrary, it follows that $K$ itself is invariant under reflection through the hyperplane $H_*$. Since reflections through affine hyperplanes containing $P+\vec T^\perp$ generate the group of all orthogonal maps fixing $P+\vec T^\perp$, this concludes the proof.
\end{proof}

\backmatter

\section{Concluding notes}\label{sec:11}

In this section, we have collected a~number of related remarks that may be of interest but would distract from the main exposition if they were placed throughout the paper. We have sorted these remarks by their corresponding section in the paper.

\subsection{sec:14}

A keen observer may notice that the orbit of any point $\boldsymbol x\in\mathbb R^4$ under the action of \eqref{eq:06} is precisely the sphere $\mathbb S^3(\mathbf 0,\|\boldsymbol x\|)$. Hence, if $K$ is invariant under the action of $\SO(3,\mathbb R)$ by representation \eqref{eq:06}, then it must be invariant also under the canonical action of $\O(4,\mathbb R)$, which contains \eqref{eq:06} as a~subgroup. In particular, \eqref{eq:06} is not a~symmetry group of any convex body $K$. So a~natural question arises:

\begin{question}
Which subgroups of $\O(\mathbb R^n)$ are symmetry groups of some convex bodies $K\subset\mathbb R^n$?
\end{question}

\noindent As the above example seems to indicate, they should have a~much simpler structure than general closed subgroups of the orthogonal group. Actually, we believe that the following conjecture is true:

\begin{conjecture}\thlabel{con:01}
Let $K\subset\mathbb R^n$, $n\geq 3$, be a~convex body. Then the symmetry group $G<\O(\mathbb R^n)$ of $K$ can be written as a~set-theoretic union
$$G=A_1\cup A_2\cup\cdots\cup A_s\cup B_1\cup B_2\cup\cdots\cup B_t,$$
where $A_i\leq G$ is a~subgroup conjugate to $\pi_{k_i}$, $i=1,2,\ldots,s$ and $B_i\leq G$ is a~subgroup conjugate to $\rho_{k_i}$, $i=1,2,\ldots,t$.
\end{conjecture}

\noindent In particular, this would mean that the two families of representations from \thref{def:04} are undoubtedly the most interesting ones. But \thref{con:01} has much more far-reaching consequences. Without loss of generality, we may assume that $G\leq\O(\mathbb R^n)$. Denote by $V_i^\perp$ the hyperaxis of $k_i$-revolution for $B_i$ and suppose that $k_i=k_j$, $i\neq j$, are minimal. Because the sum is finite and $gV_i^\perp$ is the hyperaxis of $k_i$-revolution for $gB_ig^{-1}\leq G$ for every $g\in B_j$, $V_i^\perp$ must be an invariant subspace for $B_j$. Since $V_i^\perp\neq V_j^\perp$, we have $V_i^\perp\supseteq V_j$. Thus, the number of subgroups $B_i$ with minimal $k$ can be at most $\lfloor n/(n-k)\rfloor$. For $k=1$, the result was already proved by G.~Bor, L.~Hern\'andez-Lamoneda, V.~Jim\'enez de Santiago and L.~Montejano \cite[Lemma~2.3]{10.2140/gt.2021.25.2621}, and for $k=2$, it is an immediate consequence of the result proved by B.~Zawalski \cite[Lemma~4.5]{Zawalski2024}. \thref{con:01} would also imply a~positive answer to another question \cite[Question~4.8]{Zawalski2024} raised by Zawalski, with a~much better estimate.

\subsection{sec:15}

In light of our results, it is still worth considering weaker variants of the problems discussed in \sref{sec:01} in which we assume the symmetry of more sections than just those passing through a~fixed point. The following definition, phrased already in our language, generalizes the concept of a~false axis of revolution (cf. \cite[Introduction]{faor}):

\begin{definition}\thlabel{def:06}
Let $K\subset\mathbb R^n$, $n\geq 3$, be a~convex body and let $L\in\Graff_1(\mathbb R^n)$ be any affine line in the ambient space. We say that $L$ is a~\emph{pseudo-axis of (affine) $k$-reflection} (resp. \emph{$k$-revolution}), $0\leq k<n-1$, if every point $p\in L$ is a~pseudo-center of (affine) $k$-reflection (resp. $k$-revolution). Similarly, we say that $L$ is a~\emph{quasi-axis of (affine) $k$-reflection} (resp. \emph{$k$-revolution}), $0\leq k<n-1$, if every point $p\in L$ is a~quasi-center of (affine) $k$-reflection (resp. $k$-revolution).
\end{definition}

\noindent If $K$ admits a~pseudo-axis of (affine) $k$-reflection (resp. $k$-revolution), then all the intersections $K\cap H$ of $K$ with affine hyperplanes $H\in\Graff_{n-1}(\mathbb R^n)$ are bodies of (affine) $k$-reflection (resp. $k$-revolution). Clearly, if $K$ is a~body of (affine) $k$-reflection (resp. $k$-revolution), then every line contained in the hyperaxis of $k$-reflection (resp. $k$-revolution) is a~quasi-axis of $k$-reflection (resp. $k$-revolution).\\

Although it may seem greatly overdetermined, the variant of \thref{con:06} with a~pseudo-axis is not inferior to the variant with a~quasi-center in this respect, as there is no material implication between them. Similarly, the variant with a~quasi-axis is not inferior to the variant with an aligned quasi-center (cf. \fref{fig:05}).\\

\begin{figure}
\begin{tikzpicture}
\tikzset{every edge/.append style = {-{Implies}, double distance = 1.5pt}}
\node (A) at (0,-0) {shaken center};
\node (B) at (0,-2) {pseudo-center};
\node (C) at (-3,-4) {pseudo-axis};
\node (D) at (+3,-4) {quasi-center};
\node (E) at (0,-4) {quasi-axis};
\node (F) at (0,-6) {aligned quasi-center};
\path (A) edge (B) (B) edge (C) (C) edge (E) (B) edge (D) (D) edge (E) (C) edge (F) (D) edge (F);
\end{tikzpicture}
\caption{The implication graph representing relations between different variants of \thref{con:06}}
\label{fig:05}
\end{figure}

For the sake of simplicity, we have defined all the concepts in the affine space, but it is quite natural to consider pseudo- and quasi-centers also in the projective space. Needless to say, all the aforementioned problems have their projective counterparts, but little is known about them so far, and we do not intend to reflect on them here.

\subsection{sec:01}

The theorems presented in \sref{sec:01}, originally written over several centuries and from different points of view, can be readily rephrased in the unified language we propose. For example, noting that a~body of $0$-reflection is simply centrally symmetric, the False Center Theorem reads:

\begin{custom}{\thnameref{thm:11}~\ref{thm:11}}
Let $K\subset\mathbb R^n$, $n\geq 3$, be a~convex body, and let $p\in\mathbb R^n$ be any point of the ambient space. If $p$ is a~pseudo-center of $0$-reflection for $K$, then either $K$ is centrally symmetric with respect to the point $p$ or $K$ is an ellipsoid.
\end{custom}

\noindent In this formulation, \thref{thm:11} gives an affirmative to \thref{con:12} for $k=0$. Similarly, this time noting that a~body of affine $0$-revolution is simply an ellipsoid, Brunn's theorem (i.e. \thref{thm:03}) gives an affirmative answer to \thref{con:04} for $k=0$, whereas Bezdek's conjecture (i.e. \thref{con:02,con:03}) is equivalent to \thref{con:04} for $k=1$.\\

Although the original motivation for the general problem should rather be sought in classical theorems characterizing ellipsoids, the same question emerged independently in more contemporary works of R.~Gardner and V.P.~Golubyatnikov, who considered two convex bodies $K_1,K_2\subset\mathbb R^n$ with congruent sections by hyperplanes passing through the origin and asked how different these bodies can be, and which transformations of the ambient space can transform these bodies into each other (cf. \cite[Introduction]{Golubyatnikov}).\\

In his geometric monograph, Gardner infers Rogers' \thref{thm:01} \cite[Corollary~7.1.3]{gardner1995geometric} from a~more general theorem in which we assume that $K_1\cap H$ is a~translate of $K_2\cap H$ for every hyperplane $H\in\Gr_{n-1}(\mathbb R^n)$ (cf. \cite[Theorem~7.1.1]{gardner1995geometric}). Furthermore, he also infers the False Center \thref{thm:11} \cite[Corollary~7.1.10]{gardner1995geometric} from another (difficult) theorem in which we assume that $K_2$ is a~translate of $K_1$ and $K_1\cap H$ is homothetic to $K_2\cap H$ for every hyperplane $H\in\Gr_{n-1}(\mathbb R^n)$ (cf. \cite[Theorem~7.1.7]{gardner1995geometric}). For a~much more detailed account, we refer the reader to the notes after Chapter 7 \cite[Notes 7.1 and 7.2]{gardner1995geometric}.\\

Golubyatnikov, concerned with generalizing the classical S\"uss' lemma \cite[\S3]{Sss1932ZusammensetzungVE}, proved along the way that if $K_1\cap H$ is spirally similar to $K_2\cap H$ for every hyperplane $H\in\Gr_2(\mathbb R^3)$, then $K_1$ and $K_2$ are homothetic, possibly with negative homothety coefficient, provided that $K_1\cap H$ has no orientation-preserving symmetry for every hyperplane $H\in\Gr_2(\mathbb R^3)$ (cf. \cite[Theorems~2.1.3 and 3.1.4]{Golubyatnikov}). The result was further extended by M.~Angeles Alfonseca, M.~Cordier, and D.~Ryabogin to the case of $H\in\Gr_3(\mathbb R^n)$ (cf. \cite[Corollaries~1~and~2]{ACR1}) and $H\in\Gr_4(\mathbb R^n)$ (cf. \cite[Corollaries~1~and~2]{ACR2}). It is amazing that the asymmetry condition can not be removed, as was shown by the counterexamples constructed by C.M.~Petty and J.R.~McKinney \cite[(3.1.1)]{Golubyatnikov}. Otherwise, there may be no continuous selection of the rotation angle, which is essential for the proof. It is therefore the bodies with symmetric sections that pose the greatest challenge.\\

Interestingly, the counterpart of the general problem can also be found in affine differential geometry. Namely, if $K$ is a~body of affine $k$-revolution, then the Blaschke structure of $\partial K$ \cite[Definition~II.3.2]{nomizu1994affine}, being the most fundamental algebraic invariant determining the hypersurface uniquely up to an equi-affine automorphism of the ambient space \cite[\S II.8]{nomizu1994affine}, is invariant under the pointwise action of a~group affinely conjugate to $\rho_k$ (cf. \cite[Claim~3.4]{Zawalski2024}). Clearly, an affine hypersurface invariant under the pointwise action of either $\pi_0$ or $\rho_0$ must be a~quadric, which follows immediately from the classical result of H.~Maschke, G.A.~Pick and L.~Berwald \cite[Theorem~II.4.5]{nomizu1994affine}. Y.~Lu and C.~Scharlach (2005) characterized all the $3$-dimensional affine hypersurfaces invariant under the pointwise action of any admissible group except for $\O(1,\mathbb R)$ \cite[\S 3, \S 4]{Lu2005}. As it turns out, they may take form, e.g., of a~warped product of a~$2$-dimensional quadric and an arbitrary curve, thus being a~kind of \enquote{shaken} bodies of affine $1$-revolution. K.~Schoels (2013) further developed the result of Lu and Scharlach by characterizing all affine hypersurfaces of arbitrary dimension invariant under the pointwise action of $\rho_1$ \cite[Theorem~3.1]{Schoels2013}.

\subsection{sec:08}

Remarkably, G.R.~Burton (1976) emphasized that the strict convexity assumption is necessary when considering only sections close to the boundary (e.g., in \thref{thm:03}). Generalizing a~classical characterization of the ellipsoid, he showed that all the intersections $K\cap H$ of a~convex body $K$ with affine hyperplanes $H$ sufficiently close to the boundary are centrally symmetric if and only if $K$ is the Minkowski sum of a~(not necessarily $n$-dimensional) zonotope and an ($n$-dimensional) ellipsoid:

\begin{theorem}[{cf. \cite[Theorem]{Burton1976}}]
Let $K\subset\mathbb R^n$, $n\geq 3$, be a~convex body. Then for every $\boldsymbol\xi\in\mathbb S^{n-1}$ there exists $\varepsilon(\boldsymbol\xi)>0$ such that $K\cap H_{\boldsymbol\xi}^\delta$ is centrally symmetric for every $0<h_K(\boldsymbol\xi)-\delta(\boldsymbol\xi)<\varepsilon(\boldsymbol\xi)$ if and only if $K$ is the Minkowski sum of a~zonotope and an ellipsoid.
\end{theorem}

\noindent In particular, if $K$ is strictly convex, then $K$ is an ellipsoid \cite[Corollary]{Burton1976}. The function $\varepsilon:\mathbb S^{n-1}\to\mathbb R$ is by no means unique. Note that while $\varepsilon$ can be chosen to be upper semicontinuous, it may not be possible to make it continuous. Therefore, there may exist no continuous function $\delta:\mathbb S^{n-1}\to\mathbb R$ satisfying $0<h_K(\boldsymbol\xi)-\delta(\boldsymbol\xi)<\varepsilon(\boldsymbol\xi)$. It shows that the continuity assumption in Shaken False Center \thref{thm:04} is likewise necessary.\\

C.A.~Rogers rigorously proved \thref{thm:01} only for $k=2$ and claimed without proof that for $3\leq k\leq n$ the corresponding results for $k$-dimensional sections follow immediately from the proof of \cite[Theorem~4]{Rogers1965}. It is notable that in the case of orthogonal projections, the corresponding results for $k$-dimensional sections are an immediate consequence of an analogous result \cite[Theorem~2]{Rogers1965}. For completeness, we give here a~simple argument that derives \thref{thm:01} for arbitrary $3\leq k<n$ directly from the original \cite[Theorem~4]{Rogers1965}.

\begin{proof}[Proof of Rogers' remark]
Assume that all $k$-dimensional sections of $K$ passing through the origin are centrally symmetric, and let $P$ be any $2$-dimensional plane passing through the origin. Suppose that $K\cap P$ is not centrally symmetric. Then, for every subspace $H\in\Gr_{k-2}(P^\perp)$, the (unique) center of symmetry of $K\cap(H\oplus P)$ does not belong to $P$, whence its orthogonal projection onto $H$ is not at the origin. This gives rise to a~non-vanishing continuous tangent vector field on $\Gr_{k-2}(P^\perp)$, which contradicts the Poincar\'e-Hopf theorem \cite[p.~282]{do2016differential}, since the Euler characteristic of the Grassmannian is non-zero. Hence, $K\cap P$ must be centrally symmetric. Because the choice of $P$ was arbitrary, the hypothesis of \cite[Theorem~4]{Rogers1965} is satisfied.
\end{proof}

\subsection{sec:03}

In the affine variant of Bezdek's conjecture, the compactness assumption is essential. Indeed, consider the unbounded, convex set
$$K=\{(x,y,z)\in\mathbb R_+^3\mid xyz\geq 1\}.$$
A simple, yet tedious argument shows that all sections of $K$ by generic affine planes admit the affine symmetry group isomorphic to $\D_3$ (i.e., the symmetry group of a~regular triangle), while all other sections admit the affine symmetry group isomorphic to $\D_1$ (i.e., the cyclic group of order $2$). Hence, $K$ more than satisfies the hypothesis of \thref{con:03}. However, it is clearly not a~body of affine $1$-revolution, as it does not even admit an elliptical section. We have strong reasons to believe that $K$ is the only non-trivial non-compact counterexample to \thref{con:03}.\\

As it was shown in \cite[Proposition~4.5]{10.1093/imrn/rnx211}, a~symmetric convex body whose all central sections are invariant under the action of the coordinate reflection group $(\mathbb Z_2)^{n-1}$ (i.e., are unconditional) need not be either a~body of revolution nor an ellipsoid. Indeed, for every symmetric convex body $K$, central sections of $K$ and $(K^\circ+B)^\circ$ share the same groups of symmetries. Now, for any generic ellipsoid $E$, we can find a~ball $B$ such that the Minkowski sum $E+B$ is not an ellipsoid. On the other hand, such $E+B$ is clearly not a~body of revolution, but nevertheless all its projections are unconditional. The same argument shows that in Bezdek's \thref{con:02} it is not enough to consider only planes passing through a~fixed point, and demonstrates the necessity of assuming that $p$ is not the center of symmetry for $K$ in \thref{con:12,con:04}. However, it does not provide a~simple counterexample to those conjectures whose formulations do not feature a~dichotomy. For instance, if all central sections of $(E^\circ+B)^\circ$ are bodies of $k$-revolution, then all central sections of $E$ are likewise bodies of $k$-revolution, whence, by \thref{lem:07}, $E$ is an ellipsoid of $k$-revolution, and thus $(E^\circ+B)^\circ$ is likewise a~body of $k$-revolution.

\subsection{sec:02}

If, in addition to being a~body of (affine) $k$-revolution with hyperplane of (affine) revolution $T$, $K$ is also centrally symmetric with respect to the point $o$, then for every subspace $V\subseteq T$, every intersection $K\cap H$ of $K$ with an affine hyperplane $H\in\Gr_{n-1}(\mathbb R^n)+o$ is invariant under the action of a~group of (affine) reflections with respect to the hyperaxis $(V+o)\cap H$. Observe that this is precisely the case in \cite[Theorem~2~(i)]{Alfonseca}, where $n=3$ and $k=1$. However, this concept of alignment is somewhat dual to the one proposed in this paper. In \cite{Alfonseca}, the subrepresentation assumed to be trivial is $(V+o)\cap H$, while in our paper, it is the quotient representation. The dual concept emerged in the context of a~centrally symmetric body of revolution, and we believe that this is the only situation in which it occurs.\\

As if to confirm this fact, E.~Morales-Amaya recently showed in \cite[Theorem~1]{moralesamaya2025} that if $p\in\Int K$ is a~pseudo-center of $(-1)$-reflection, and there exists a~hyperplane $T\in\Gr_{n-1}(\mathbb R^n)$ such that $T+p$ is not a~hyperplane of symmetry of $K$, and every intersection $K\cap H$ of $K$ with an affine hyperplane $H\in\Gr_{n-1}(\mathbb R^n)+p$ is invariant under the action of a~group conjugate to $\pi_{n-1}$, with hyperaxis of reflection parallel to $T\cap H$, then $K$ is an ellipsoid of revolution, with axis of revolution orthogonal to $T$. In his argument, he concluded that $p$ must be a~pseudo-center of reflection, whence $K$ must be centrally symmetric, by virtue of the False Center \thref{thm:11}. Note that while $(-k)$-aligned quasi-centers and pseudo-centers of $k$-reflection exist for an arbitrary body of $k$-revolution, their dual counterparts exist (for now only conjecturally) exclusively for a~centrally symmetric body of $k$-revolution and for an ellipsoid of $k$-revolution, respectively.

\subsection{sec:09}

\thref{thm:16,thm:02} do not hold in dimension $n=3$. Consider the following counterexample: Let $K$ be the unit Euclidean ball centered at $p$, let $T=\langle\boldsymbol e_3\rangle^\perp$ be the horizontal plane, and let $\sigma:\Gr_1(T)\to\Gr_2(\mathbb R^3)$ be any smooth map such that $\sigma(H)\cap T=H$ for all $H\in\Gr_1(T)$. Every intersection $K\cap(\sigma(H)+p)$ is invariant under reflection through the axis that is the orthogonal projection of $\langle\boldsymbol e_3\rangle+p$ onto $\sigma(H)+p$. Now, consider an open slab $S$ parallel to $T$ such that $C\cap S\subset K\cap S$, where $C$ is the blunt right circular cone with axis $\langle\boldsymbol e_3\rangle$, arising from \thref{lem:08}. Through each point of $\partial K\cap S$ passes exactly one plane from $\sigma(\Gr_1(T))+p$, whence the shape of $\partial K\cap S$ is uniquely determined by the axes of symmetry and a~function $w:(\tau_1,\tau_2)\times\Gr_1(T)\to\mathbb R$ given by
$$w(\tau,H)\colonequals\Vol_1(K\cap(T+\tau\boldsymbol e_3)\cap(\sigma(H)+p)).$$
By adding to $w(\tau,H)$ a~compactly supported smooth perturbation that is small in $C^2$ norm, we obtain a~parametrization of another strongly convex body $K'$ whose selected sections share the same axes of symmetry with the sections of $K$, but $K'$ is not a~body of $1$-reflection (neither orthogonal nor affine).\\

In the proof of \thref{thm:02}, instead of \thref{thm:05} requiring a~function $\sigma$ defined on the Grassmannian $\Gr_{n-2}(T)$, we used \thref{thm:20} requiring a~function $\sigma$ defined on the \emph{unoriented} Grassmannian $\Gr^+_{n-2}(T)$. Note that since the unoriented Grassmannian $\Gr_2(T)$ is a~quotient space of the oriented Grassmannian $\Gr^+_2(T)$, by the universal property of the quotient map, every continuous map $\Gr_{n-2}(T)\to\Gr_{n-1}(\mathbb R^n)$ descends from a~unique continuous map $\Gr^+_{n-2}(T)\to\Gr_{n-1}(\mathbb R^n)$, but not the other way around. Therefore, the hypothesis of \thref{thm:02} could potentially be relaxed if only we did not use the topological properties of the unoriented Grassmannian later in the proof. In particular, we needed the key fact that such a~family of hyperplanes sweeps the entire boundary $\partial K$. Without it, the most we can hope for is that there exists a~body of (affine) $1$-revolution $L$ such that $K\cap(\sigma(H)+p)=L\cap(\sigma(H)+p)$ for every hyperplane $H\in\Gr_{n-2}(T)$.

\subsection{sec:05}

In the last section devoted to applications, S.P.~Olovyanishnikov proved that if $f:\mathbb R^n\to\mathbb R$, $n\geq 3$, is a~probability density function whose restriction to every straight line is symmetric and unimodal, then either $f$ is a~normalized indicator function of a~compact convex set or all the level sets of $f$ are concentric homothetic ellipsoids \cite[Prilozheniya~I]{Olov}. Note that this result follows immediately from \thref{lem:03}, which may therefore be viewed as its $2$-dimensional counterpart. Interestingly, Olovyanishnikov remarks that his proof does not work for $n=2$, in which case the theorem still holds, but requires a~more subtle argument \cite[Zamechaniye]{Olov}. For analytic curves, he refers to \cite[\S 17]{blaschke1923vorlesungen}, but he also claims that for general closed curves, it can be shown without any regularity assumption. However, we were unable to find any such proof in the literature. \thref{con:07} is a~strengthening of Olovyanishnikov's result.\\

Note that the conclusion of \thref{con:07} is clearly not true if, instead of all the affine hyperplanes, we consider an arbitrary $1$-codimensional family. Indeed, if both $K_1$ and $K_2$ are centrally symmetric with respect to the origin, then the hypothesis is satisfied for all the affine hyperplanes passing through the origin. Moreover, if $K_1$ is the body of flotation for $K_2$, then the hypothesis is satisfied for all the affine hyperplanes supporting $K_1$. On the other hand, if $K_0$ is a~body of flotation for $K_1$, then by \cite[Lemma and subsequent Zamechaniye~1 and Zamechaniye~2]{Olov} it is also a~body of flotation for $K_2$, which implies that it should be enough to consider the family of all the affine hyperplanes supporting $K_0$, as long as the following conjecture is true:

\begin{conjecture}
Let $K_1\subsetneq K_2\subsetneq K_3\subset\mathbb R^n$, $n\geq 2$, be pairwise different convex bodies. If $K_1$ is a~body of flotation for $K_2$ and $K_3$, and $K_2$ is a~body of flotation for $K_3$, then $K_1,K_2,K_3$ are concentric, homothetic ellipsoids.
\end{conjecture}

\noindent To the best of our knowledge, this is an open problem.\\

In the statement of \thref{lem:05}, the assumption that the family $\mathcal F$ contains an affine hyperplane in every direction can not be easily relaxed. Indeed, consider, e.g., $\boldsymbol A=\Diag(1,1,\ldots,1,\lambda)$, $|\lambda|<1$, and $\boldsymbol b=(0,0,\ldots,0,\delta)$, $\delta\neq 0$. Then for every $\boldsymbol\xi\in\mathbb S^{n-1}$ such that $\xi_n\neq 0$ we have
$$s(\boldsymbol\xi)\colonequals\frac{\delta}{(1-\lambda^2)\xi_n}\equalscolon t(\boldsymbol\xi)$$
satisfying \eqref{eq:18}. Moreover, the corresponding affine hyperplanes $H_{\boldsymbol\xi}^t$ intersect both $K_1$ and $K_2$ for every $\boldsymbol\xi$ such that $|\xi_n|>|\delta|(1-\lambda^2)^{-1}$, which may be an arbitrarily large open subset of $\mathbb S^{n-1}$.

\subsection{sec:17}

In both arguments, the use of the reference section $K_{\tau_*}$ was crucial. Unless all the hyperplanes in the range of $\sigma$ contain the diametral chord of $K$ passing through $p$, we can always find a~section $K_\tau$ that is disjoint from a~non-empty open subset of $\sigma(\Gr_{n-2}(T))+p$. Therefore, we are almost never able to apply any known result to each convex body $K_\tau$, $\tau\in\mathbb R$, individually.\\

Observe, however, that if $p\in\partial K$ is a~point on the boundary of $K$ and $\partial K$ is locally strongly convex at $p$, then the intersections $K\cap T$ of $K$ with affine hyperplanes $T$ converge in Banach-Mazur distance to $\mathbb B^{n-1}$ as $T$ converges to the hyperplane tangent to $K$ at $p$. Therefore, in order to find a~single reference section $K_{\tau_*}$ admitting a~certain group of symmetries, it is enough to look near the unit Euclidean ball. This opens up the possibility of using techniques of harmonic analysis proposed by M.~Angeles Alfonseca, F.~Nazarov, D.~Ryabogin, and V.~Yaskin in \cite{ANRY}, and at the same time provides an example of a~natural setting in which the usually problematic locality assumption is not restrictive.

\section*{Acknowledgments}

We would like to thank F.~Nazarov, D.~Ryabogin, and V.~Yaskin for many inspiring discussions and invaluable comments that helped us improve the paper, E.~Morales-Amaya for bringing to our attention the dual concept of alignment, and M.~Wojciechowski for suggesting to us the way to prove \thref{lem:06}.

\bibliography{references}{}
\bibliographystyle{amsplain}

\end{document}